\documentclass[11pt]{amsart}
\usepackage{graphicx,color,ulem,overpic}
\usepackage{graphics}
\usepackage{amsmath}
\usepackage{amscd}
\usepackage{latexsym}
\usepackage[all]{xy}
\begin{document}
\textwidth 5.5in
\textheight 8.3in
\evensidemargin .75in
\oddsidemargin.75in

\newtheorem{quest}{Question}
\newtheorem{thm}[quest]{Theorem}
\newtheorem{facts}[quest]{Fact}
\newtheorem{lem}[quest]{Lemma}
\newtheorem{defi}[quest]{Definition}
\newtheorem{conj}[quest]{Conjecture}
\newtheorem{cor}[quest]{Corollary}
\newtheorem{prop}[quest]{Proposition}
\newtheorem{prob}[quest]{Problem}
\newtheorem{claim}[quest]{Claim}
\newtheorem{exm}[quest]{Example}
\newtheorem{cond}{Condition}
\newtheorem{rmk}[quest]{Remark}
\newtheorem{que}[quest]{Question}
\newcommand{\p}[3]{\Phi_{p,#1}^{#2}(#3)}
\def\tu{\widetilde{\Upsilon}}
\def\Z{\mathbb Z}
\def\N{\mathbb N}
\def\C{\mathcal{C}}
\def\D{\mathcal{D}}
\def\R{\mathbb R}
\def\g{\overline{g}}
\def\odots{\reflectbox{\text{$\ddots$}}}
\newcommand{\tg}{\overline{g}}
\def\proof{{\bf Proof. }}
\def\ee{\epsilon_1'}
\def\ef{\epsilon_2'}
\title{Upsilon invariants of L-space cable knots}
\author{Motoo Tange}
\thanks{The author is supported by JSPS KAKENHI Grant Number 26800031}
\address{Institute of Mathematics, University of Tsukuba,
 1-1-1 Tennodai, Tsukuba, Ibaraki 305-8571, Japan}
\email{tange@math.tsukuba.ac.jp}
\date{\today}
\maketitle
\begin{abstract}
We give a formula of the Upsilon invariant of any L-space cable knot $K_{p,q}$ using $p,\Upsilon_K$ and $\Upsilon_{T_{p,q}}$.
The integral value of the Upsilon invariant gives a ${\mathbb Q}$-valued knot concordance invariant.
We compute the integral values for L-space iterated cable knots.
\end{abstract}
%
%
  \section{Introduction}
  \label{intro}
\subsection{Knot concordance invariants}
Let $K$ be a knot in $S^3$.
Let $S$ be a Seifert matrix of $K$ and $\omega$ a complex number with $|\omega|=1$.
The {\it Tristram-Levine signature} ({\it TL-signature}) $\sigma_{K}(\omega)$ is defined as the signature of the matrix
$$(1-\omega)S+(1-\bar{\omega})S^T.$$
The usual definition of the {\it knot signature} $\sigma(K)$ implies $\sigma_K(-1)$, hence we have $\sigma(K)=\sigma_K(-1)$.
Thus the TL-signature is a refinement of $\sigma(K)$.
It is classically well-known that $\sigma(K)$ can give a lower bound of the 4-ball genus of $K$.

By using the knot filtration of the knot Floer chain complex $CFK^\infty(S^3,K)$,
Ozsv\'ath and Szab\'o defined a knot concordance invariant $\tau:\mathcal{C}\to {\mathbb Z}$ (the {\it $\tau$-invariant}), where $\mathcal{C}$ is a knot concordance group.
In fact, $\tau(K)$ also has a similar lower bound for 4-ball genus, as mentioned in \cite{OS5}.
The estimate by $\tau(K)$ is sharper than the one by $\sigma(K)$.
In \cite{OSS}, Ozsv\'ath, Stipsicz and Szab\'o defined a knot concordance invariant ({\it $\Upsilon$-invariant}) $\Upsilon:\mathcal{C}\to C([0,2])$ ($K\mapsto \Upsilon_K$), where $C([0,2])$ is the group consisting of continuous functions over the closed interval $[0,2]$.
Livingston in \cite{L} gave a simpler definition of $\Upsilon_K$.
The $\Upsilon$-invariant is defined essentially by using the doubly graded filtration of $CFK^\infty(S^3,K)$ and is regarded as a refinement of the $\tau$-invariant.
In fact, $\tau(K)=-\Upsilon_K'(0)$ holds.
$\Upsilon$-invariant has been also applied to finding knot concordance classes linearly independent mutually, for example, as seen in \cite{Ch} and \cite{OSS}.

We have seen that the invariant $\tau(K)$ is a Heegaard Floer analog of $\sigma$-invariant.
The TL-signature $\sigma_K(\omega)$ and $\Upsilon_K'(t)$ are locally constant away from zeros of $\Delta_K(t)$, and refinements of $\sigma(K)$ and $\tau(K)$ respectively.
We can schematically show these relationships as below (\ref{sche}):
\begin{equation}
\label{sche}
\xymatrix{
\sigma(K) \ar[rr]^{\text{HF analog}}\ar@![d]_{\text{refinement}}  && \tau(K) \ar@![d]^{\text{refinement}}\\
 \sigma_K(\omega) \ar[rr]_{\text{HF analog}}&& -\Upsilon'_K(t)}
\end{equation}
The minus before $\Upsilon'_K(t)$ is due to the restriction $-\Upsilon'_K(0)=\tau(K)$.

\subsection{Cabling formula of invariants}
Consider cabling formulas for several invariants.
Let $K$ be a knot in $S^3$.
Let $V$ be a tubular neighborhood of $K$.
For integers $p,q$, the $(p,q)$-cable $K_{p,q}$ of $K$ is defined to be the simple closed curves on $\partial V$ whose homology class is
$p\cdot {\bf l}+q\cdot {\bf m}$ in $H_1(\partial V)$, where ${\bf l}$ and ${\bf m}$ are classes represented by a longitude curve and a meridian curve on $\partial V$.
If $p,q$ are coprime integers, then $K_{p,q}$ is a knot and we call it a {\it $(p,q)$-cable knot}.
The cabling formula for Alexander polynomial is as follows:
\begin{equation}\Delta_{K_{p,q}}(t)=\Delta_{K}(t^p)\Delta_{T_{p,q}}(t).\label{cablingAlex}\end{equation}
Since the knot Floer homology is a categorification of the Alexander polynomial, 
it is natural to try to find the cabling formula for the knot Floer homology.
For example, as in \cite{Hed2} and \cite{Hed} many studies have been done.
However, it has not been completely succeeded yet.
In general, it is difficult to give the cabling formula for the knot Floer homology. 

Due to \cite{Li}, the cabling formula of the TL-signature is known as follows:
\begin{equation}
\label{tlsig}
\sigma_{K_{p,q}}(\omega)=\sigma_K(\omega^p)+\sigma_{T_{p,q}}(\omega).
\end{equation}
These cabling formulas for $\sigma_K(\omega)$ and $\Delta_K(t)$ both consist of the invariants of the companion knot $K$ and the ones of torus knots.

Here we introduce Hom's cabling formula for the $\tau$-invariant.
This formula uses additional information $\epsilon(K)$ to compute $\tau$.
The cabling formula for knot Floer chain complex, if any, would be much more complicated than classical invariants $\sigma(K)$ or $\Delta_K(t)$.
We state it here.
\begin{thm}[Hom \cite{Hom2}]
Let $K\subset S^3$.
Then $\tau(K_{p,q})$ is completely determined by $p,q,\tau(K)$, and $\epsilon(K)$ in the following manner.
\begin{enumerate}
\item[(1)] If $\epsilon(K)=1$, then $\tau(K_{p,q})=p\tau(K)+(p-1)(q-1)/2$.
\item[(2)] If $\epsilon(K)=-1$, then $\tau(K_{p,q})=p\tau(K)+(p-1)(q+1)/2$.
\item[(3)] If $\epsilon(K)=0$, then $\tau(K)=0$ and\\
$\tau(K_{p,q})=\tau(T_{p,q})=\begin{cases}(p-1)(q+1)/2&q<0\\(p-1)(q-1)/2&q>0.\end{cases}$
\end{enumerate}
\end{thm}

\subsection{Motivation}
Chen in \cite{Ch} gives an inequality for the $\Upsilon$-invariant of any cable knot.
However, it is also hard to obtain a closed formula of $\Upsilon_{K_{p,q}}$.
To do so in small cases, we focus on any {\it L-space knot}, which is defined as a knot $K$ whose positive surgery of $K$ is an L-space.
Here a rational homology sphere $Y$ is an {\it L-space}, if $Y$ has the same Heegaard Floer homology as that of $S^3$ for any spin$^c$ structure on $Y$.
According to \cite{OS3}, the knot Floer homology of any L-space knot is simple.

Here we recall a necessary and sufficient condition for a cable knot $K_{p,q}$ to be an L-space knot by Hedden and Hom.
\begin{thm}[Hedden \cite{Hed} and Hom \cite{Hom}]
\label{HeHocond}
Let $K$ be a knot with the Seifert genus $g$.
$K_{p,q}$ is an L-space knot if and only if $K$ is an L-space knot with $(2g-1)p\le q$.
\end{thm}

Therefore, we reach the following natural problem.
\begin{prob}
Find the cabling formula of the $\Upsilon$-invariant on L-space knots.
\end{prob}
We shall consider the $\Upsilon$-invariant of any L-space cable knot (Theorem \ref{main}, \ref{cor} and \ref{s}).
We, first, give a cabling formula of L-space cable knots $K$ with $2gp\le q$
with the Seifert genus $g$.
After that, we consider a cabling formula  in the case of $(2g-1) p< q<2gp$.


\subsection{Cabling formula for $\Upsilon$ of L-space knots with $2gp\le q$.}
We give a cabling formula of $\Upsilon$ for $2gp\le q$.
\begin{thm}[The case of $2gp\le q$]
\label{main}
Let $K$ be an L-space knot with the Seifert genus $g$.
Let $p,q$ be relatively prime positive integers with $2gp\le q$.
Then the $\Upsilon$-invariant of $K_{p,q}$ is computed as follows:
\begin{equation}\label{cab}
\Upsilon_{K_{p,q}}(t)=\Upsilon_{K}(s)+\Upsilon_{T_{p,q}}(t),\end{equation}
where $s$ is the real number with $s\equiv pt(\bmod 2)$ and $0\le s\le 2$.
\end{thm}
Here $\Upsilon_K(s)$ part in this formula can be regarded as a $p$-fold amalgamated function of $\Upsilon_K(t)$ in terms of a function in $t$.
Here the amalgamated function means the deformation as in {\sc Figure} \ref{ama}.
\begin{figure}[htbp]
\begin{center}
\includegraphics{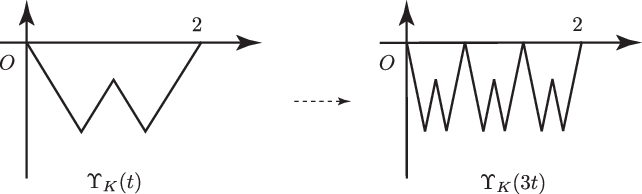}
\caption{The amalgamated function of $3$-copies of $\Upsilon_K(t)$.}\label{ama}
\end{center}
\end{figure}
In other words, by seeing $\Upsilon_K$ as a function over ${\mathbb R}/2{\mathbb Z}$, 
the $p$-fold amalgamated function means $\Upsilon_K(pt)$.
This formula (\ref{cab}) is similar to the cabling formula (\ref{tlsig}).

To prove this formula, we use a simple expression of $\Upsilon_K(t)$ for any L-space knot $K$ by Borodzik and Livingston in \cite{BL}.
They wrote down an $\Upsilon$-invariant formula for any L-space knot $K$ by using the formal semigroup $S_K$ as follows.
\begin{prop}[\cite{BL}]
\label{formulaBL}
Let $K$ be an L-space knot with genus $g$.
Then for any $t \in  [0, 2]$ we have 
$$\Upsilon_K(t)= \underset{m\in \{0,\cdots,2g\}}{\max}\{-2\#(S_K\cap [0,m))-t(g-m)\}.$$
\end{prop}
We will explain the formal semigroup $S_K$ in Section~\ref{fs}.
In this paper we put the following function:
$$\tu_{K}(t,m)=-2\#(S_K\cap [0,m))-t(g-m).$$
Hence the $\Upsilon$-invariant is written as $\Upsilon_K(t)=\underset{m\in \{0,\cdots,2g\}}{\max}\tu_K(t,m)$.

%


\subsection{Cabling formula for $\Upsilon$ of L-space knots with $(2g-1)p< q< 2gp$.}
\label{2formula}
We set
$$\mu_K:=\underset{0< m< 2g}\min\frac{2\#(S_K\cap [0,m))}{m},$$
$$\delta:=q-(2g-1)p.$$
and
\begin{equation}
\label{zdefinition}
z^{i_2}_{i_1}=i_1p+i_2q.
\end{equation}

Here for any real number $t$ with $2i/p\le t\le 2(i+1)/p$ we define $\Upsilon^{\delta,1}_{p,q}(t)$ and $\Upsilon^{\delta,2}_{p,q}(t)$ to be
$$\Upsilon^{\delta,1}_{p,q}(t)=\underset{z^i_0-\delta<m\le z^i_0}{\max}\tu_{T_{p,q}}(t,m),\ \Upsilon^{\delta,2}_{p,q}(t)=\underset{z^i_{-1}<m\le z^i_0-\delta}{\max}\tu_{T_{p,q}}(t,m).$$
Then,
$$\max\left\{\Upsilon_{p,q}^{\delta,1}(t),\Upsilon_{p,q}^{\delta,2}(t)\right\}=\Upsilon_{T_{p,q}}(t)$$
holds.
We have $\Upsilon_{T_{p,q}}(t)=\underset{z^i_{-1}<m\le z^i_0}\max\tu_{T_{p,q}}(t,m)$, due to Lemma \ref{firstcase}.
For any L-space knot $K$ we define the {\it truncated $\Upsilon$-invariant} as follows:
$$\Upsilon_{K}^{tr}(s)=\max_{\nu\in\{1,2,\cdots,2g-1\}}\tu_K(s,\nu).$$ 
Here we state the second main theorem in this article.
\begin{thm}
\label{cor}
Let $K$ be an L-space knot with the Seifert genus $g$.
Let $p,q$ be relatively prime integers with $(2g-1)p<q<2gp$.
Let $t$ be a real number with $0\le t\le 2$.
Suppose that $s$ is a real number and $i$ is an integer with the property that
$pt\equiv s\bmod 2$ and $0\le s\le 2$ and $i=(pt-s)/2$.

Further, suppose that $s$ satisfies either of the following conditions:
$$\begin{cases}
0\le s\le 2-\mu_K&i=0\\
\mu_K\le s\le 2-\mu_K&1\le i\le p-2\\
\mu_K\le s\le 2&i=p-1.
\end{cases}$$
Then 
$$\Upsilon_{K_{p,q}}(t)=\Upsilon_K(s)+\Upsilon_{T_{p,q}}(t)$$
holds.
\end{thm}
We clarify the invariant $\Upsilon_{K_{p,q}}(t)$ in the remaining case.
\begin{thm}
\label{s}
Let $K, g, p,q, t, s$, and $i$ be parameters satisfying the condition in the first paragraph in Theorem~\ref{cor}.

Further, suppose that $s$ satisfies either of 
$$\begin{cases}
0\le s< \mu_K&1\le i\le p-1\\
2-\mu_K < s\le 2&0\le i\le p-2.
\end{cases}$$
In the former case, the following is satisfied:
\begin{equation}
\label{alter}
\Upsilon_{K_{p,q}}(t)=\max\left\{\Upsilon_K(s)+\Upsilon_{p,q}^{\delta,1}(t),\Upsilon_{K}^{tr}(s)+\Upsilon^{\delta,2}_{p,q}(t)\right\}.
\end{equation}
In the latter case, the following is satisfied:
\begin{equation}
\label{alter2}
\Upsilon_{K_{p,q}}(t)=\max\left\{\Upsilon_K(s)+\Upsilon_{p,q}^{\delta,1}(2-t),\Upsilon_{K}^{tr}(s)+\Upsilon_{p,q}^{\delta,2}(2-t)\right\}.
\end{equation}
\end{thm}
Actually, this formula holds in the case of Theorem~\ref{cor}.
Then, $\Upsilon_K^{tr}(s)$ is equal to $\Upsilon_{K}(s)$ (Lemma~\ref{23}), hence, as a result, the formula in Theorem~\ref{cor} holds.
Here we give an inequality of L-space cabling formula for $\Upsilon$-invariant.  
\begin{cor}
\label{pureinequality}
Let $K$ be an L-space knot with the Seifert genus $g$.
We assume that $(2g-1)p<q<2gp$.
Let $t$ be a real number with $0\le t\le 2$, and $i$ and $s$ an integer and a real number with $2i/p\le t\le2(i+1)/p$, $2i+s=pt$ and $0\le s\le 2$.
Then 
$$\Upsilon_{T_{p,q}}(t)+\Upsilon_{K}^{tr}(s)\le  \Upsilon_{K_{p,q}}(t)\le \Upsilon_{T_{p,q}}(t)+\Upsilon_K(s)$$
holds.
\end{cor}
In particular if $\mu_K\le s\le 2-\mu_K$, then the inequalities become the equalities.

\subsection{Example $(T_{3,7})_{3,35}$.}
\label{exin}
Here we verify the formulas in Theorem~\ref{cor} and \ref{s} in the case of $(2g-1)p< q<2gp$.
Consider the $(3,35)$-cable knot of $K=T_{3,7}$.
Then $p=3$, $q=35$, and $g(K)=6$ hold.
Therefore $K_{3,35}$ is an L-space knot from the Hedden-Hom criterion.
We compare the functions $\Upsilon_{K_{3,35}}(t)-\Upsilon_{T_{3,35}}(t)$ and $\Upsilon_K(3t)$.
Here we compute the two functions with the aid of Mathematica program by \cite{KK}.
See {\sc Figure} \ref{38}.

The value $\mu_K$ defined above is $2/3$.
Let $t$ and $s$ be real numbers with $0\le t,s\le 2$ and $pt\equiv s\bmod 2$ and $i$ an integer $i=(pt-s)/2$.

If $0\le i\le 2$, $2i+s=3t$ and 
$$\begin{cases}0\le s\le 4/3&i=0\\
2/3\le s\le 4/3&i=1\\
2/3\le s\le 2&i=2,\end{cases}$$
then $\Upsilon_{K_{3,35}}(t)=\Upsilon_{T_{3,35}}(t)+\Upsilon_{K}(s)$ holds,
as described in {\sc Figure}~\ref{38}.

On the other hand, for the remaining regions, e.g., $i=1$ and $0<s<2/3$ or $4/3<s<2$, the $\Upsilon_{K_{3,35}}(t)$ violates the formula (\ref{cab}).
In Section~\ref{examplecompute}, we try to compute some of the actual functions of $\Upsilon_{K_{3,35}}(t)$ over the following regions:
$i=1$ and $0<s<2/3$ and $i=0$ and $4/3<s<2$.

\begin{figure}[htbp]
\includegraphics[width=.6\textwidth]{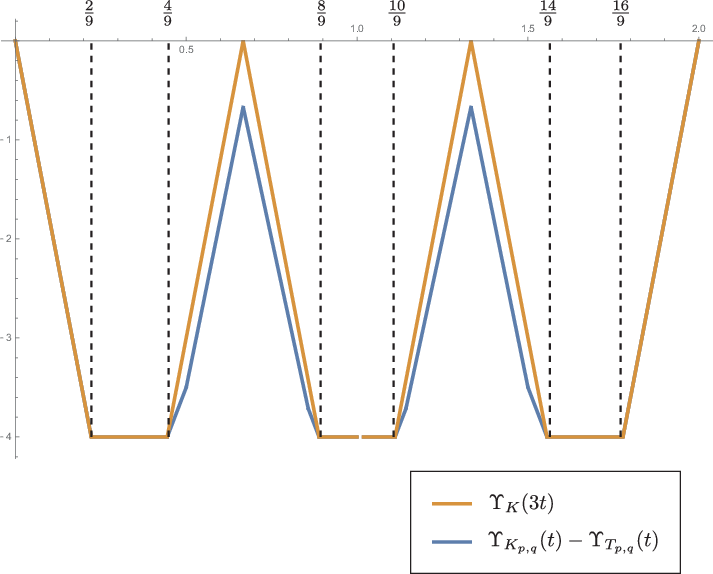}
\caption{The red graph is $\Upsilon_{K}(3t)$. The blue graph is the different part of $\Upsilon_{K_{3,35}}(t)-\Upsilon_{T_{3,35}}(t)$ from $\Upsilon_K(3t)$.}
\label{38}
\end{figure}

\subsection{Integral value of $\Upsilon_K(t)$ on $[0,2]$.}
Here we propose a knot concordance invariant that it is easy to compute from the cabling formula.
We compute the integral value of $\Upsilon_K(t)$ on the interval $[0,2]$:
$$I(K)=\int_0^2\Upsilon_K(t)dt,$$
which is also a knot concordance invariant.
The motivation of this value is inspired by the $S^1$-integral value $\int_{S^1}\sigma_K(\omega)$.
The $S^1$-integral value $\int_{S^1}\sigma_{T_{p,q}}(\omega)$ is computed as follows:
$$\int_{S^1}\sigma_{T_{p,q}}(\omega)=-\frac{1}{3}\left(pq-\frac{p}{q}-\frac{q}{p}+\frac{1}{pq}\right)=4(s(q,p)+s(p,q)-s(1,pq)),$$
where the function $s$ is the Dedekind sum.
This computation has been done by many topologists for example \cite{KM}, \cite{Nem}, \cite{B} and \cite{Co}.

On the other hand, $I(T_{p,q})$ is computed as follows:
\begin{prop}
\label{torusknot}
Let $p,q$ be relatively prime positive integers.
Let $a_i\ge 0$ be the $i$-th term of the non-negative continued fraction of $q/p$:
\begin{equation}
\label{continued}
q/p=a_1+\frac{1}{a_2+\frac{1}{\cdots+\frac{1}{a_n}}}=:[a_1,\cdots,a_n].
\end{equation}
Then we have
$$I(T_{p,q})=-\frac{1}{3}(pq-\sum_{i=1}^na_i).$$
\end{prop}

Let ${\bf p}_i$ be a pair $(p_i,q_i)$ of coprime integers.
We give a formula of $I$ for iterated cable L-space knots $K({\bf p}_1,\cdots,  {\bf p}_n):=(\cdots(K_{p_1,q_1})_{p_2,q_2}\cdots)_{p_n,q_n}$.

\begin{thm}
\label{cabling}
Let $K:=L_0$ be an L-space knot and $L$ an iterated cable knot $K({\bf p}_1,\cdots,  {\bf p}_n)$.
We put $L_i:=K({\bf p}_1,\cdots,  {\bf p}_i)$ for any integer $1\le i\le n$.
If ${\bf p}_i=(p_i,q_i)$ satisfies $2g(L_{i-1})p_i\le q_i$ for any integer $i$ with $1\le i\le n$, then the integral $I(K({\bf p}_1,\cdots,  {\bf p}_n))$ is computed as follows:
$$I(K({\bf p}_1,\cdots,  {\bf p}_n))=I(K)+\sum_{i=1}^nI(T_{p_i,q_i}).$$
\end{thm}
For the $S^1$-integral value of $\sigma_L(\omega)$ of iterated torus knots $L$, a similar formula holds.
See \cite{B}.

In Theorem~\ref{cabling} we deal with L-space iterated torus knots $L=T_{p,q}({\bf p}_1,\cdots, {\bf p}_n)$ satisfying $2g(L_{i-1})p_i\le q_i$ for any integer $i$ with $1\le i\le n$..
Are these knots different from general L-space iterated torus knots?
As an application of Theorem~\ref{cabling}, we prove that there exists an L-space iterated torus knot
that is not knot concordant to any L-space iterated torus knots satisfying $2g(L_{i-1})p_i\le q_i$ (Proposition~\ref{2527}).

\begin{rmk}
On the other hand, in terms of the diagram (\ref{sche}), we might consider integral values
$-\int_0^1\Upsilon_K'(s)ds=-\Upsilon_K(1)$ or $\int_{S^1}\Sigma(s)ds$,
where $\Sigma(s)$ is a piecewise linear continuous function with $\Sigma(0)=0$ and $\Sigma'(\omega)=\sigma(\omega)$
for $\omega\in S^1$ away from finite points.
\end{rmk}
\section*{Acknowledgements}
This work was started by computing the integral values of the $\Upsilon$-invariants of any torus knots.
The author thanks for Min Hoon Kim.
He told me the $\Upsilon$-invariant formula for the torus knots and the reference \cite{FK}.
This became my motivation to compute the $\Upsilon$-invariants of the L-space cable knots.
Furthermore, he gave me many useful comments for my earlier manuscript.
The author would like to thank an anonymous referee for indicating several unclear points.

\section{Preliminaries}
In this section we introduce tools to prove our main theorem (Theorem~\ref{main}).
%
\subsection{Formal semigroup}
\label{fs}
Let $K$ be an L-space knot with the Seifert genus $g(K)=g$.
Expanding the rational function $\Delta_K(t)/(1-t)$ as follows:
$$\frac{\Delta_K(t)}{1-t}=\sum_{s\in S_K}t^{s},$$
we obtain a subset $S_K\subset {\mathbb Z}_{\ge 0}$.
Hence, the coefficients of the right hand side are $0$ or $1$.
This subset $S_K$ is called the {\it formal semigroup of $K$}.
The following properties hold:\\
{\bf Fundamental facts:}
\begin{itemize}
\item Any algebraic knot is an L-space knot. If $K$ is an algebraic knot, then $S_K$ is a semigroup (by \cite{Wal}).
\item $S_K\cap {\mathbb Z}_{\ge 2g}={\mathbb Z}_{\ge 2g}$.
\item $s\in S_K\Leftrightarrow 2g(K)-1-s\not\in S_K$.
\item There is a cabling formula for formal semigroup (Proposition~\ref{wangformula}).
\end{itemize}

For example, if $K$ is a right-handed torus knot $T_{p,q}$, then $S_{T_{p,q}}$ is the semigroup generated by the positive integers $p,q$, namely, $S_{T_{p,q}}=\langle p,q\rangle=\{pa+qb\mid a,b\in {\mathbb Z}_{\ge 0}\}$ holds.
There exists an L-space but not-algebraic knot.
For example, for $n\ge 1$ the $(-2,3,2n+1)$ pretzel knot $K_n$ is an L-space knot, and the formal semigroup is as follows:
$$S_{K_n}=\{0,3,5,7,\cdots,2n-1,2n+1,2n+2\}\cup {\mathbb Z}_{\ge 2n+4}.$$
$K_1=T_{3,4}$ and $K_2=T_{3,5}$ are only two algebraic knots in this sequence.
For $n\ge 3$, we can easily see that $S_{K_n}$ is not a semigroup.
The Alexander polynomials of $(-2,3,2n+1)$-pretzel knots can be found, for example,  in \cite{Hi}.

Wang, in \cite{Wang}, proved the cabling formula for the formal semigroup of any L-space knot as follows:
\begin{prop}[A cabling formula for formal semigroup \cite{Wang}]
\label{wangformula}
Let $K$ be a nontrivial L-space knot.
Suppose $p\ge 2$ and $p(2g(K)-1)\le q$.
Then $S_{K_{p,q}} = pS_K+q{\mathbb Z}_{\ge 0}:=\{pa+qb\mid a\in S_{K},b\in {\mathbb Z}_{\ge 0}\}$.
\end{prop}
Hence, $S_{K_{p,q}}=pS_K+q{\mathbb Z}_{\ge 0}$ can be decomposed as follows: 
\begin{equation}
\label{dec}
S_{K_{p,q}}=pS_K\cup (pS_K+q)\cup (pS_K+2q)\cup \cdots, 
\end{equation}
where $pS+x:=\{pa+x|a\in S\}$ for a set $S\subset{\mathbb R}$.

Here we prove the following lemma.
\begin{lem}
\label{second}
Let $S_K$ be a formal semigroup of a non-trivial L-space knot $K$.
Then $1\not\in S_K$ holds.
\end{lem}
\begin{proof}
If $1\in S_K$, then the Alexander polynomial of the L-space knot is computed as follows:
$$\Delta_K(t)=(1-t)(1+t+t^sf(t))=1-t^2+t^s(1-t)f(t),$$
where $s\ge 2$ and $f(t)$ is a series.
Thus the coefficient of $t$ in $\Delta_K(t)$ vanishes.
The coefficient of $t$ of the Alexander polynomial of a non-trivial L-space knot is $-1$ due to \cite{H}.
Thus $K$ must be the trivial knot.
\hfill$\Box$
\end{proof}
In the case of lens space knots, there would be some restrictions to $S_K$.
The results in \cite{T} can give some restrictions.

Here we claim that $\mu_K\le  1$ if $K$ is a non-trivial knot, where $\mu_K$ is defined in Section~\ref{2formula}.
For, from the third in the fundamental facts above and Lemma~\ref{second}, we have 
\begin{equation}
\label{ineq10}
\mu_K\le \frac{2\#(S_K\cap [0,2))}{2}=1.
\end{equation}
\subsection{Proof of Theorem~\ref{main}.}

Let $K$ be an L-space knot with the Seifert genus $g$.
Throughout this section we assume that the relatively prime positive integers $p,q$ satisfy $2gp\le q$.
In particular, $K_{p,q}$ is also an L-space knot.

For any L-space knot $K$ we put 
$$\varphi_K(m)=\#(S_{K}\cap [0,m)),$$
and
$$\Phi_K(t,m)=\varphi_K(m)-tm/2.$$
Here we prove the following lemma.
\begin{lem}
Let $\nu$ be an integer with $0\le \nu\le 2g$ and $p$ a positive integer.
Then we have 
\begin{equation}
\label{fundlem1}
\varphi_K(2g-\nu)=g-\nu+\varphi_K(\nu)\\
\end{equation}
\begin{equation}
\label{fundlem2}
\Phi_{K}(t,m+p)-\Phi_{K}(t,m)=\#(S_{K}\cap [m,m+p))-tp/2.
\end{equation}
\end{lem}
\begin{proof}
Let $\bar{S}_K$ be the complement of $S_K$ in ${\mathbb Z}$.
(\ref{fundlem1}) and (\ref{fundlem2}) are due to the following equalities:
\begin{eqnarray*}
\varphi_K(2g-\nu)&=&g-\#(S_K\cap [2g-\nu,2g))\\
&=&g-\#(\bar{S}_K\cap [0,\nu))\\
&=&g-\nu+\varphi_K(\nu),
\end{eqnarray*}
\begin{eqnarray*}
\Phi_{K}(t,m+p)-\Phi_{K}(t,m)&=&
\varphi_K(m+p)-t(m+p)/2-\varphi_K(m)+tm/2\\
&=&\#(S_{K}\cap [m,m+p))-tp/2.
\end{eqnarray*}
\qed\end{proof}

According to Proposition \ref{formulaBL}, the $\Upsilon$-invariant of an L-space knot $K$ is rewritten as follows:
\begin{eqnarray}
\Upsilon_{K}(t)&=&-2\underset{m\in \{0,1,\cdots,2g\}}{\min}\{\varphi_{K}(m)-tm/2\}-tg(K).\nonumber\\
&=&-2\underset{m\in \{0,1,\cdots,2g\}}{\min}\Phi_K(t,m)-tg(K).\label{upsilonPhi}
\end{eqnarray}
Extending the function $\varphi_K(m)$ as $\varphi_K(m)\equiv0$ if $m<0$, we can define $\Phi_K(t,m)$ over $m\in {\mathbb Z}$.
We note that the function $\Phi_K(t,m)$ satisfies the following:
$$\Phi_K(t,m)=\begin{cases}-tm/2& m<0\\(1-t/2)m-g&m>2g.\end{cases}$$
The last row is due to the third fundamental fact in Section~\ref{fs}.
In other words, this fact means that since the exact half of ${\mathbb Z}\cap [0,2g)$ is included in $S_K\cap [0,2g)$, we have $\#(S_K\cap [0,2g))=\#({\mathbb Z}\cap [0,2g))-g=g$.

Thus, if a subset $S\subset {\mathbb Z}$ includes $\{0,1,\cdots,2g\}$ then we have 
$$\underset{m\in S}\min\Phi_K(t,m)=\underset{m\in\{0,1,\cdots2g\}}\min\Phi_K(t,m).$$

The genus $g(K_{p,q})=:g_{p,q}$ is equal to the degree of $\Delta_{K_{p,q}}(t)$ since $K_{p,q}$ is an L-space knot.
Thus from the cabling formula (\ref{cablingAlex}), we have
$$g(K_{p,q})=pg+g_{p,q}.$$
We denote $\varphi_{K_{p,q}}(m)$ by $\varphi(m)$ and $\Phi_{K_{p,q}}(t,m)$ by $\Phi(t,m)$.

Here we give a first setting to prove cabling formulas.\\
\underline{Setting:}
\begin{equation}
\label{setting}
\begin{cases}
p: \text{ an integer,}\\
t,s:\text{ real numbers with }0\le t,s\le 2\text{ and }pt\equiv s\bmod 2,\\
i:\text{ an integer with }i=(pt-s)/2.
\end{cases}
\end{equation}
Here we recall the definition of $z^{i_2}_{i_1}$ in (\ref{zdefinition}).
\begin{lem}
\label{firstcase}
Let $K$ be an L-space knot with $g=g(K)$.
Let $p,q$ be relatively prime integers with $2gp\le q$.
Let $t,s,i$ be parameters satisfying (\ref{setting}).
Then we have
$$\underset{0\le m\le 2g(K_{p,q})}\min\Phi(t,m)=\underset{z^i_{-1}< m\le z^i_{2g}}\min\Phi(t,m).$$
\end{lem}
\begin{proof}
Take parameters $t,s,i$ satisfying (\ref{setting}).
The parameter $t$ is fixed here.
We prove the following claim:
\begin{claim}
\label{mm+p1}
If $m\le z^i_{-1}(\Leftrightarrow m+p\le z^{i}_0)$ holds, then $\#(S_{K_{p,q}}\cap[m,m+p))\le i$ holds.
\end{claim}

\begin{proof}
Using the decomposition (\ref{dec}) right after Proposition~\ref{wangformula} we obtain
{\sc Figure}~\ref{graph1}.
It describes the local picture of $S_{K_{p,q}}$ around $m=z^i_0$.
The top line in {\sc Figure}~\ref{graph1} consists of 
the components of $pS_K+iq$ and the second top line consists of $pS_K+(i-1)q$.
The other elements $\cup\{pS_K+jq|j=0,\cdots,i-2\}$ are omitted in {\sc Figure}~\ref{graph1}.
The shaded circles in {\sc Figure}~\ref{graph1} present $S_{K_{p,q}}$
as points projected to the $m$-axis.
The empty circles correspond to $(p{\mathbb Z}_{\ge 0}+q{\mathbb Z}_{\ge 0})\setminus S_{K_{p,q}}$.

\begin{figure}[htbp]
{\unitlength 0.1in%
\begin{picture}(38.9500,11.0500)(10.1500,-21.8500)%
%
\special{pn 8}%
\special{pa 1015 2060}%
\special{pa 4910 2060}%
\special{fp}%
\special{sh 1}%
\special{pa 4910 2060}%
\special{pa 4843 2040}%
\special{pa 4857 2060}%
\special{pa 4843 2080}%
\special{pa 4910 2060}%
\special{fp}%
\put(44.4000,-19.9000){\makebox(0,0)[lb]{$m$}}%
%
\special{pn 8}%
\special{ar 1600 1400 50 50 0.0000000 6.2831853}%
%
\special{pn 8}%
\special{ar 2400 1400 50 50 0.0000000 6.2831853}%
%
\special{pn 4}%
\special{pa 2440 1370}%
\special{pa 2370 1440}%
\special{fp}%
\special{pa 2450 1390}%
\special{pa 2390 1450}%
\special{fp}%
\special{pa 2420 1360}%
\special{pa 2360 1420}%
\special{fp}%
\special{pa 2400 1350}%
\special{pa 2350 1400}%
\special{fp}%
%
\special{pn 8}%
\special{ar 2990 1130 50 50 0.0000000 6.2831853}%
%
\special{pn 8}%
\special{ar 3790 1130 50 50 0.0000000 6.2831853}%
%
\special{pn 4}%
\special{pa 3040 1120}%
\special{pa 2980 1180}%
\special{fp}%
\special{pa 3030 1100}%
\special{pa 2960 1170}%
\special{fp}%
\special{pa 3010 1090}%
\special{pa 2950 1150}%
\special{fp}%
\special{pa 2990 1080}%
\special{pa 2940 1130}%
\special{fp}%
%
\special{pn 8}%
\special{pa 3785 1130}%
\special{pa 3785 2070}%
\special{dt 0.045}%
%
\special{pn 8}%
\special{pa 2985 1130}%
\special{pa 2985 2070}%
\special{dt 0.045}%
%
\special{pn 8}%
\special{pa 2400 1400}%
\special{pa 2400 2070}%
\special{dt 0.045}%
%
\special{pn 8}%
\special{pa 1600 1400}%
\special{pa 1600 2070}%
\special{dt 0.045}%
\put(37.4000,-22.5000){\makebox(0,0){$z^i_1$}}%
\put(29.4000,-22.5000){\makebox(0,0){$z^{i}_0$}}%
\put(16.0000,-22.4000){\makebox(0,0){$z^{i-1}_{2g-1}$}}%
\put(24.0000,-22.4000){\makebox(0,0){$z^{i-1}_{2g}$}}%
%
\special{pn 8}%
\special{pa 2990 1135}%
\special{pa 4240 1135}%
\special{fp}%
%
\special{pn 8}%
\special{pa 1390 1400}%
\special{pa 2915 1400}%
\special{fp}%
%
\special{pn 8}%
\special{pa 1600 2070}%
\special{pa 1600 2150}%
\special{fp}%
%
\special{pn 8}%
\special{ar 3200 1400 50 50 0.0000000 6.2831853}%
%
\special{pn 4}%
\special{pa 3240 1370}%
\special{pa 3170 1440}%
\special{fp}%
\special{pa 3250 1390}%
\special{pa 3190 1450}%
\special{fp}%
\special{pa 3220 1360}%
\special{pa 3160 1420}%
\special{fp}%
\special{pa 3200 1350}%
\special{pa 3150 1400}%
\special{fp}%
%
\special{pn 8}%
\special{pa 3200 1400}%
\special{pa 3200 2070}%
\special{dt 0.045}%
%
\special{pn 8}%
\special{ar 4000 1400 50 50 0.0000000 6.2831853}%
%
\special{pn 4}%
\special{pa 4040 1370}%
\special{pa 3970 1440}%
\special{fp}%
\special{pa 4050 1390}%
\special{pa 3990 1450}%
\special{fp}%
\special{pa 4020 1360}%
\special{pa 3960 1420}%
\special{fp}%
\special{pa 4000 1350}%
\special{pa 3950 1400}%
\special{fp}%
%
\special{pn 8}%
\special{pa 4000 1400}%
\special{pa 4000 2070}%
\special{dt 0.045}%
%
\special{pn 8}%
\special{pa 2915 1400}%
\special{pa 3150 1400}%
\special{dt 0.045}%
%
\special{pn 8}%
\special{pa 3170 1400}%
\special{pa 4485 1400}%
\special{fp}%
%
\special{pn 8}%
\special{pa 4465 1400}%
\special{pa 4795 1400}%
\special{dt 0.045}%
%
\special{pn 8}%
\special{pa 1025 1400}%
\special{pa 1455 1400}%
\special{dt 0.045}%
%
\special{pn 8}%
\special{pa 4240 1135}%
\special{pa 4545 1135}%
\special{dt 0.045}%
%
\special{pn 8}%
\special{pa 2400 2060}%
\special{pa 2400 2140}%
\special{fp}%
%
\special{pn 8}%
\special{pa 2980 2060}%
\special{pa 2980 2140}%
\special{fp}%
%
\special{pn 8}%
\special{pa 3200 2060}%
\special{pa 3200 2140}%
\special{fp}%
%
\special{pn 8}%
\special{pa 4000 2060}%
\special{pa 4000 2140}%
\special{fp}%
%
\special{pn 8}%
\special{pa 3780 2060}%
\special{pa 3780 2140}%
\special{fp}%
\put(32.4000,-22.5000){\makebox(0,0){$z^{i-1}_{2g+1}$}}%
\put(40.4000,-22.5000){\makebox(0,0){$z^{i-1}_{2g+2}$}}%
\end{picture}}%
\caption{The elements in $pS_K+(i-1)q$ and $pS_K+iq$ for $0\le j\le i$.}
\label{graph1}
\end{figure}

In the case of $m\le z^i_{-1}$, since $(pS_K+jq)\cap [m,m+p)$ is at most one point, we have
$$\#(S_{K_{p,q}}\cap [m,m+p))=\sum_{j=0}^{i-1}\#((pS_{K}+jq)\cap [m,m+p))\le i.$$
\qed
\end{proof}

Suppose that $m$ is an integer with $z^i_{2g-1}< m\le 2g(K_{p,q})$.
Similar to Claim \ref{mm+p1}, we prove the following claim:
\begin{claim}
\label{mm+p2}
If $z^i_{2g-1}<m\le 2g(K_{p,q})$, then $\#(S_{K_{p,q}}\cap [m,m+p))\ge i+1$ holds.
\end{claim}
\begin{proof}
If $z^i_{2g-1}<m$, then for a non-negative integer $j$ with $j\le i$, 
$(pS_K+jq)\cap [m,m+p)$ is exact one point
from the second of the fundamental facts in Section~\ref{fs}.
Thus 
$$\#(S_{K_{p,q}}\cap [m,m+p))\ge \sum_{j=0}^i\#((pS_K+jq)\cap [m,m+p))=i+1$$
holds.
\hfill$\Box$
\end{proof}
We go back to the proof of Lemma~\ref{firstcase}.
In the case of $m\le z^i_{-1}$, using Claim~\ref{mm+p1} and (\ref{fundlem2}) we have $\Phi(t,m+p)-\Phi(t,m)\le i-tp/2\le 0$.
In the case of $z^i_{2g-1}<m\le 2g(K_{p,q})$, using Claim~\ref{mm+p2} and (\ref{fundlem2}), we have
$\Phi(t,m+p)-\Phi(t,m)=\#(S_{K_{p,q}}\cap [m,m+p))-tp/2\ge i+1-tp/2\ge 0$.

Thus the minimum value of $\Phi(t,m)$ over $0\le m\le 2g(K_{p,q})$ is equal to 
the minimum value over $z^i_{-1}<m\le z^i_{2g}$.
\hfill$\Box$
\end{proof}

As a corollary of this lemma, it follows that if $K$ is the unknot, then we have
\begin{equation}
\label{eq23}
\underset{0\le m\le 2g_{p,q}}\min\Phi_{T_{p,q}}(t,m)=\underset{z^i_{-1}< m\le z^i_0}\min\Phi_{T_{p,q}}(t,m).
\end{equation}

Next, we investigate the minimum value of $\Phi(t,m)$ in the region
$$I_i={\mathbb Z}\cap (z^i_{-1},z^i_{2g}].$$
We prepare the following claim to prove Theorem~\ref{main}:
\begin{claim}
\label{minmin}
The minimum value of $\Phi(t,m)$ over $I_i$ coincides with 
$$\underset{\nu\in S_K,\nu-1\not\in S_K}{\min}\left\{\underset{z^i_{\nu-1}< m\le z^i_{\nu}}{\min}\Phi(t,m)\right\}.$$
\end{claim}
\begin{proof}
The behavior of $\Phi(t,m)$ among $I_i$ becomes {\sc Figure}~\ref{type1}.
Let $\omega_\nu$ be the minimum value of $\Phi(t,m)$ among $z^i_{\nu-1}<m\le z^i_{\nu}$.
If $\nu\not\in S_K$, then $\Phi(t,m+p)-\Phi(t,m)=\#(S_{K_{p,q}}\cap [m,m+p))-tp/2= i-tp/2\le 0$, i.e., $\omega_{\nu+1}\le \omega_{\nu}$ holds.
If $\nu\in S_K$, then 
$\Phi(t,m)-\Phi(t,m-p)=\#(S_{K_{p,q}}\cap [m-p,m))-tp/2= i+1-tp/2\ge 0$, i.e., $\omega_{\nu}\le \omega_{\nu+1}$ holds.
Hence, we have only to consider the minimum values of $\omega_\nu$ in the case of 
$\nu\in S_K$ and $\nu-1\not\in S_K$.\qed
\begin{figure}[htbp]
{\unitlength 0.1in%
\begin{picture}(41.2000,14.1500)(2.7000,-29.1500)%
%
\special{pn 8}%
\special{pa 600 2811}%
\special{pa 4390 2811}%
\special{fp}%
\special{sh 1}%
\special{pa 4390 2811}%
\special{pa 4323 2791}%
\special{pa 4337 2811}%
\special{pa 4323 2831}%
\special{pa 4390 2811}%
\special{fp}%
%
\special{pn 4}%
\special{sh 1}%
\special{ar 1435 2030 8 8 0 6.2831853}%
%
\special{pn 8}%
\special{pa 2000 2800}%
\special{pa 2000 2000}%
\special{dt 0.045}%
%
\special{pn 8}%
\special{pa 2400 2086}%
\special{pa 2400 2800}%
\special{dt 0.045}%
%
\special{pn 8}%
\special{pa 2800 2798}%
\special{pa 2800 1990}%
\special{dt 0.045}%
%
\special{pn 8}%
\special{pa 3200 1976}%
\special{pa 3200 2794}%
\special{dt 0.045}%
%
\special{pn 8}%
\special{pa 3600 2798}%
\special{pa 3600 2186}%
\special{dt 0.045}%
%
\special{pn 8}%
\special{pa 1440 2025}%
\special{pa 1842 1920}%
\special{dt 0.045}%
%
\special{pn 8}%
\special{pa 1840 1920}%
\special{pa 2240 2120}%
\special{dt 0.045}%
%
\special{pn 8}%
\special{pa 2265 2125}%
\special{pa 2686 2015}%
\special{dt 0.045}%
%
\special{pn 8}%
\special{pa 2670 2020}%
\special{pa 3062 1918}%
\special{dt 0.045}%
%
\special{pn 8}%
\special{pa 3070 1915}%
\special{pa 3470 2115}%
\special{dt 0.045}%
%
\special{pn 8}%
\special{pa 3470 2115}%
\special{pa 3870 2315}%
\special{dt 0.045}%
%
\special{pn 4}%
\special{sh 1}%
\special{ar 3065 1925 8 8 0 6.2831853}%
%
\special{pn 4}%
\special{sh 1}%
\special{ar 3890 2320 16 16 0 6.2831853}%
\put(12.0000,-29.8000){\makebox(0,0){$z^i_{-1}$}}%
\put(16.0000,-29.8000){\makebox(0,0){$z^i_0$}}%
\put(42.5200,-27.0400){\makebox(0,0){$m$}}%
%
\special{pn 8}%
\special{pa 1200 2800}%
\special{pa 1200 2870}%
\special{fp}%
%
\special{pn 8}%
\special{pa 680 2900}%
\special{pa 680 1500}%
\special{fp}%
\special{sh 1}%
\special{pa 680 1500}%
\special{pa 660 1567}%
\special{pa 680 1553}%
\special{pa 700 1567}%
\special{pa 680 1500}%
\special{fp}%
%
\special{pn 8}%
\special{pa 1435 2030}%
\special{pa 680 2030}%
\special{dt 0.045}%
%
\special{pn 8}%
\special{ar 2248 2130 53 53 0.0000000 6.2831853}%
%
\special{pn 8}%
\special{ar 1432 2032 53 53 0.0000000 6.2831853}%
%
\special{pn 8}%
\special{pa 3890 2326}%
\special{pa 4282 2224}%
\special{dt 0.045}%
%
\special{pn 8}%
\special{ar 3890 2325 53 53 0.0000000 6.2831853}%
%
\special{pn 8}%
\special{pa 1600 2800}%
\special{pa 1600 2000}%
\special{dt 0.045}%
%
\special{pn 8}%
\special{pa 1200 2800}%
\special{pa 1200 1920}%
\special{dt 0.045}%
%
\special{pn 8}%
\special{pa 1040 1830}%
\special{pa 1440 2030}%
\special{dt 0.045}%
\put(20.0000,-29.8000){\makebox(0,0){$z^i_1$}}%
%
\special{pn 8}%
\special{pa 2000 2800}%
\special{pa 2000 2870}%
\special{fp}%
\put(24.0000,-29.8000){\makebox(0,0){$z^i_2$}}%
\put(4.7000,-20.4000){\makebox(0,0){$\mu_i$}}%
%
\special{pn 4}%
\special{sh 1}%
\special{ar 1840 1925 16 16 0 6.2831853}%
%
\special{pn 4}%
\special{sh 1}%
\special{ar 2240 2125 16 16 0 6.2831853}%
%
\special{pn 4}%
\special{sh 1}%
\special{ar 1434 2028 16 16 0 6.2831853}%
%
\special{pn 4}%
\special{sh 1}%
\special{ar 2624 2028 16 16 0 6.2831853}%
%
\special{pn 4}%
\special{sh 1}%
\special{ar 3070 1916 16 16 0 6.2831853}%
%
\special{pn 4}%
\special{sh 1}%
\special{ar 3486 2122 16 16 0 6.2831853}%
%
\special{pn 8}%
\special{pa 4000 2790}%
\special{pa 4000 2280}%
\special{dt 0.045}%
%
\special{pn 8}%
\special{pa 1600 2810}%
\special{pa 1600 2880}%
\special{fp}%
%
\special{pn 8}%
\special{pa 2400 2810}%
\special{pa 2400 2880}%
\special{fp}%
\end{picture}}%
\caption{The places of local minimum points of $\Phi(t,m)$ over $m\in I_i$.}
\label{type1}
\end{figure}
\end{proof}

This minimum value over $I_i$ is:
\begin{equation}\label{ladder}\min\left\{\sum_{l=0}^mp\left(\frac{i+\epsilon(l)+1}{p}-\frac{t}{2}\right)+\mu_i|m=-1,0,1,2,\cdots,2g-1\right\}\end{equation}
where $\mu_i$ is the minimum value of $\Phi(t,m)$ over $(z^i_{-1},z^i_0]$.
The function $\epsilon(\nu)$ is defined as follows:
$$\epsilon(\nu)=\begin{cases}0&\nu\in S_K\\-1&\nu\not\in S_K\end{cases}$$
Here, the case of $m=-1$ for the summation in the minimum (\ref{ladder}) means that the sum is $0$.
Since $\sum_{l=0}^m(\epsilon(l)+1)=\#(S_K\cap [0,m+1))$ holds, the summation part in (\ref{ladder}) is computed as follows:
\begin{eqnarray*}
&&\underset{-1\le m\le 2g-1}{\min}\left\{\#(S_K\cap [0,m+1))-\left(\frac{tp}{2}-i\right)(m+1)\right\}\\
&=&\underset{0\le m\le 2g}{\min}\left\{\#(S_K\cap [0,m))-sm/2\right\}=\underset{0\le m\le 2g}\min\Phi_K(s,m).
\end{eqnarray*}
Then we have 
\begin{equation}
\label{phiformula}
\underset{m\in I_i}{\min}\Phi(t,m)=\underset{0\le m\le 2g}\min\Phi_K(s,m)+\mu_i.
\end{equation}
Hence, we obtain
\begin{eqnarray*}
\Upsilon_{K_{p,q}}(t)&=&-2\underset{0\le m\le 2g}{\min}\Phi_K(s,m)-2\mu_i-tg(K_{p,q})\\
&=&\Upsilon_K(s)+sg-2\mu_i-t(pg+g_{p,q}).
\end{eqnarray*}


Here we claim the following.
$\Phi_{p,q}(t,m)$ means $\Phi_{T_{p,q}}(t,m)$.
\begin{claim}
\label{mucom}
We have
$$\mu_i=\underset{z^i_{-1}< m\le z^i_0}{\min}\Phi_{p,q}(t,m)-ig.$$
\end{claim}
\begin{proof}
Let $S_K^c$ be the set $\{n\in {\mathbb Z}_{\ge 0}|n\not\in S_K\}$.
$$S_{K_{p,q}}=pS_K\cup (pS_K+q)\cup(pS_K+2q)\cup\cdots $$
is constructed by removing $pS_K^c\cup(pS_K^c+q)\cup(pS_K^c+2q)\cup\cdots $ from 
$S_{T_{p,q}}=p{\mathbb Z}_{\ge 0}\cup(p{\mathbb Z}_{\ge 0}+q)\cup(p{\mathbb Z}_{\ge 0}+2q)\cup\cdots$.
Since $\#S_K^c=g$ holds, 
$\#(pS_K^c\cup(pS_K^c+q)\cup(pS_K^c+2q)\cup\cdots \cup(pS_K^c+(i-1)q))=ig$.
Hence, $\mu_i=\underset{z^i_{-1}<m\le z^i_{0}}{\min}\Phi(t,m)=\underset{z^i_{-1}<m\le z^i_{0}}{\min}\Phi_{p,q}(t,m)-ig$.\qed
\end{proof}

Hence, we go back to the proof of Theorem~\ref{main}.
Using (\ref{eq23}) we have,
\begin{eqnarray}
\mu_i&=&\underset{z^i_{-1}< m\le z^i_{0}}{\min}\Phi_{p,q}(t,m)-ig\nonumber\\
&=&\underset{0\le m\le 2g_{p,q}}{\min}\Phi_{p,q}(t,m)-ig.\label{m}
\end{eqnarray}

Therefore by using the formula (\ref{upsilonPhi}), we obtain the following:
\begin{eqnarray*}
\Upsilon_{K_{p,q}}(t)&=&\Upsilon_K(s)+sg-2\underset{0\le m\le 2g_{p,q}}{\min}\Phi_{p,q}(t,m)+2ig-(pg+g_{p,q})t\\
&=&\Upsilon_K(s)+\Upsilon_{T_{p,q}}(t).
\end{eqnarray*}
\qed
\section{The case of $(2g-1)p<q<2gp$.}
\subsection{Minimum value of $\Phi(t,m)$.}
Let $K$ be an L-space knot with the Seifert genus $g$.
Throughout this section we assume that $p,q$ are relatively prime positive integers satisfying $(2g-1)p<q<2gp$.
In particular, $K_{p,q}$ is an L-space knot.
We consider $\Phi(t,m)=\#(S_{K_{p,q}}\cap [0,m))-tm/2$.
Setting $\delta=q-(2g-1)p$, we have $0<\delta<p$.

We put 
$$I_i^{\delta}:={\mathbb Z}\cap (z^i_0-\delta, z^{i+1}_0].$$
This $I^\delta_i$ is slightly smaller than $I_i$ and is included in $I_i$.
Here $z_0^i-\delta=z^{i-1}_{2g-1}$.

In the case of $(2g-1)p<q<2gp$, the behavior of $S_{K_{p,q}}$ around $z^i_0$ is different from the case of $2gp\le q$ as in {\sc Figure}~\ref{graph2}.
\begin{figure}[htbp]
{\unitlength 0.1in%
\begin{picture}(30.2000,11.0500)(7.8000,-21.8500)%
%
\special{pn 8}%
\special{pa 780 2070}%
\special{pa 3800 2070}%
\special{fp}%
\special{sh 1}%
\special{pa 3800 2070}%
\special{pa 3733 2050}%
\special{pa 3747 2070}%
\special{pa 3733 2090}%
\special{pa 3800 2070}%
\special{fp}%
\put(35.0000,-20.4000){\makebox(0,0)[lb]{$m$}}%
%
\special{pn 8}%
\special{ar 1600 1400 50 50 0.0000000 6.2831853}%
%
\special{pn 8}%
\special{ar 2400 1400 50 50 0.0000000 6.2831853}%
%
\special{pn 4}%
\special{pa 2440 1370}%
\special{pa 2370 1440}%
\special{fp}%
\special{pa 2450 1390}%
\special{pa 2390 1450}%
\special{fp}%
\special{pa 2420 1360}%
\special{pa 2360 1420}%
\special{fp}%
\special{pa 2400 1350}%
\special{pa 2350 1400}%
\special{fp}%
%
\special{pn 8}%
\special{ar 1890 1130 50 50 0.0000000 6.2831853}%
%
\special{pn 8}%
\special{ar 2690 1130 50 50 0.0000000 6.2831853}%
%
\special{pn 4}%
\special{pa 1940 1120}%
\special{pa 1880 1180}%
\special{fp}%
\special{pa 1930 1100}%
\special{pa 1860 1170}%
\special{fp}%
\special{pa 1910 1090}%
\special{pa 1850 1150}%
\special{fp}%
\special{pa 1890 1080}%
\special{pa 1840 1130}%
\special{fp}%
%
\special{pn 8}%
\special{pa 2685 1130}%
\special{pa 2685 2070}%
\special{dt 0.045}%
%
\special{pn 8}%
\special{pa 1885 1130}%
\special{pa 1885 2070}%
\special{dt 0.045}%
%
\special{pn 8}%
\special{pa 2400 1400}%
\special{pa 2400 2070}%
\special{dt 0.045}%
%
\special{pn 8}%
\special{pa 1600 1400}%
\special{pa 1600 2070}%
\special{dt 0.045}%
\put(26.8000,-22.4000){\makebox(0,0){$z^i_1$}}%
\put(18.8000,-22.4000){\makebox(0,0){$z^i_0$}}%
\put(15.0000,-22.5000){\makebox(0,0){$z^{i}_{0}-\delta$}}%
\put(23.0000,-22.4000){\makebox(0,0){$z^{i}_{1}-\delta$}}%
%
\special{pn 8}%
\special{pa 1890 1130}%
\special{pa 3200 1130}%
\special{fp}%
%
\special{pn 8}%
\special{pa 1190 1410}%
\special{pa 3190 1410}%
\special{fp}%
%
\special{pn 8}%
\special{pa 1600 2070}%
\special{pa 1600 2120}%
\special{fp}%
%
\special{pn 8}%
\special{pa 3190 1410}%
\special{pa 3690 1410}%
\special{dt 0.045}%
%
\special{pn 8}%
\special{pa 3220 1130}%
\special{pa 3700 1130}%
\special{dt 0.045}%
%
\special{pn 8}%
\special{pa 850 1410}%
\special{pa 1350 1410}%
\special{dt 0.045}%
%
\special{pn 8}%
\special{pa 1880 2070}%
\special{pa 1880 2120}%
\special{fp}%
%
\special{pn 8}%
\special{pa 2680 2070}%
\special{pa 2680 2120}%
\special{fp}%
%
\special{pn 8}%
\special{pa 2400 2070}%
\special{pa 2400 2120}%
\special{fp}%
\end{picture}}%
\caption{The behavior of $pS_K+(i-1)q$ and $pS_K+iq$ around $z_0^i$.}
\label{graph2}
\end{figure}

Here we prove the following lemma analogous to Lemma~\ref{firstcase}.
\begin{lem}
\label{secondcase}
Let $t,s$, and $i$ be parameters satisfying (\ref{setting}).
Then we have
$$\underset{0\le m\le 2g(K_{p,q})}\min\Phi(t,m)=\underset{m\in I_i^\delta}\min\Phi(t,m).$$
\end{lem}
\begin{proof}
In the same way as Lemma \ref{firstcase},
if $m$ satisfies $m\le z^i_0-\delta$, then $\varphi(t,m+p)-\varphi(t,m)=\#(S_{K_{p,q}}\cap [m,m+p))\le i$ holds.
Hence we have, using (\ref{fundlem2}),
$$\Phi(t,m+p)-\Phi(t,m)\le i-tp/2\le 0.$$

If $m$ satisfies $z^{i+1}_{-1}<m\le 2g(K_{p,q})-p$, then we can compute as follows:
$\varphi(m+p)-\varphi(m)=\#(S_{K_{p,q}}\cap [m,m+p))\ge i+1$.
Hence we have 
$$\Phi(t,m+p)-\Phi(t,m)\ge i+1-tp/2\ge 0.$$
Thus the minimum value of $\Phi(t,m)$ coincides with the minimum value over $I_i^\delta$.
\qed\end{proof}

\begin{figure}[htbp]
{\unitlength 0.1in%
\begin{picture}(39.7000,15.2200)(6.0000,-29.6200)%
%
\special{pn 8}%
\special{pa 600 2800}%
\special{pa 4570 2800}%
\special{fp}%
\special{sh 1}%
\special{pa 4570 2800}%
\special{pa 4503 2780}%
\special{pa 4517 2800}%
\special{pa 4503 2820}%
\special{pa 4570 2800}%
\special{fp}%
%
\special{pn 8}%
\special{pa 1500 2000}%
\special{pa 1500 2800}%
\special{dt 0.045}%
%
\special{pn 8}%
\special{pa 1900 2800}%
\special{pa 1900 1890}%
\special{dt 0.045}%
%
\special{pn 8}%
\special{pa 2300 1835}%
\special{pa 2300 2800}%
\special{dt 0.045}%
%
\special{pn 8}%
\special{pa 2700 2800}%
\special{pa 2700 2035}%
\special{dt 0.045}%
%
\special{pn 8}%
\special{pa 3100 2190}%
\special{pa 3100 2800}%
\special{dt 0.045}%
%
\special{pn 8}%
\special{pa 3500 2800}%
\special{pa 3500 2090}%
\special{dt 0.045}%
%
\special{pn 8}%
\special{pa 1045 1815}%
\special{pa 1445 2015}%
\special{dt 0.045}%
%
\special{pn 8}%
\special{pa 1440 2010}%
\special{pa 1861 1900}%
\special{dt 0.045}%
%
\special{pn 8}%
\special{pa 1840 1910}%
\special{pa 2261 1800}%
\special{dt 0.045}%
%
\special{pn 8}%
\special{pa 2250 1810}%
\special{pa 2650 2010}%
\special{dt 0.045}%
%
\special{pn 8}%
\special{pa 2650 2010}%
\special{pa 3050 2210}%
\special{dt 0.045}%
%
\special{pn 4}%
\special{sh 1}%
\special{ar 1445 2010 8 8 0 6.2831853}%
%
\special{pn 4}%
\special{sh 1}%
\special{ar 1845 1910 8 8 0 6.2831853}%
%
\special{pn 4}%
\special{sh 1}%
\special{ar 2245 1810 8 8 0 6.2831853}%
%
\special{pn 4}%
\special{sh 1}%
\special{ar 2645 2010 8 8 0 6.2831853}%
%
\special{pn 4}%
\special{sh 1}%
\special{ar 3045 2210 8 8 0 6.2831853}%
%
\special{pn 8}%
\special{pa 3055 2205}%
\special{pa 3993 1960}%
\special{dt 0.045}%
%
\special{pn 4}%
\special{sh 1}%
\special{ar 3445 2110 8 8 0 6.2831853}%
\put(44.4000,-27.0000){\makebox(0,0){$m$}}%
\put(12.9000,-30.3500){\makebox(0,0){$z^{i-1}_{2g-1}$}}%
\put(15.9000,-30.3500){\makebox(0,0){$z^i_0$}}%
%
\special{pn 8}%
\special{pa 1500 2800}%
\special{pa 1500 2850}%
\special{fp}%
%
\special{pn 8}%
\special{pa 1200 1895}%
\special{pa 1200 2800}%
\special{dt 0.045}%
\put(19.3000,-30.3500){\makebox(0,0){$z^i_1$}}%
\put(35.0000,-30.3500){\makebox(0,0){$z^i_{2g-1}$}}%
%
\special{pn 8}%
\special{pa 820 2870}%
\special{pa 820 1440}%
\special{fp}%
\special{sh 1}%
\special{pa 820 1440}%
\special{pa 800 1507}%
\special{pa 820 1493}%
\special{pa 840 1507}%
\special{pa 820 1440}%
\special{fp}%
\put(38.3500,-30.3500){\makebox(0,0){$z^{i+1}_0$}}%
%
\special{pn 8}%
\special{pa 3800 2800}%
\special{pa 3800 2015}%
\special{dt 0.045}%
%
\special{pn 8}%
\special{ar 1449 2014 59 59 1.0993364 1.1147243}%
%
\special{pn 8}%
\special{pa 1100 1850}%
\special{pa 1100 2800}%
\special{dt 0.045}%
\put(9.4500,-30.3500){\makebox(0,0){$z^i_{-1}$}}%
%
\special{pn 8}%
\special{pa 1100 2800}%
\special{pa 1100 2850}%
\special{fp}%
%
\special{pn 8}%
\special{pa 1900 2800}%
\special{pa 1900 2845}%
\special{fp}%
%
\special{pn 8}%
\special{pa 3800 2800}%
\special{pa 3800 2850}%
\special{fp}%
%
\special{pn 8}%
\special{pa 3500 2800}%
\special{pa 3500 2855}%
\special{fp}%
%
\special{pn 8}%
\special{ar 3040 2215 59 59 1.0993364 1.1147243}%
%
\special{pn 8}%
\special{pa 1200 2800}%
\special{pa 1200 2850}%
\special{fp}%
%
\special{pn 4}%
\special{pa 1098 2852}%
\special{pa 916 2914}%
\special{fp}%
%
\special{pn 4}%
\special{pa 1202 2852}%
\special{pa 1278 2912}%
\special{fp}%
\end{picture}}%
\caption{A schematic picture of local minimum points of $\Phi(t,m)$ over $m\in (z^{i-1}_{2g-1},z^i_{2g-1}]$.}
\label{type2}
\end{figure}
\section{Proof of Theorem \ref{cor}}
Let $p,q$ be relatively prime positive integers with $(2g-1)p <q< 2gp$.
Let $t,s$, and $i$ be parameters satisfying (\ref{setting}).
We decompose the case of $(2g-1)p<q<2gp$ into the following two parts: 
\begin{equation}
\label{fise}
\begin{cases}
0\le s\le 2-\mu_K &i=0\\
\mu_K\le s\le 2-\mu_K&1\le i\le p-2\\
\mu_k\le s\le 2&i=p-1.
\end{cases}
\end{equation}
\begin{equation}
\label{fise2}
\begin{cases}
0\le s< \mu_K&1\le i\le p-1\\
2-\mu_K < s\le 2&0\le i\le p-2.
\end{cases}
\end{equation}
Actually, this decomposition corresponds to whether the difference of $\Upsilon_{K_{p,q}}(t)$
and $\Upsilon_{K}(s)+\Upsilon_{T_{p,q}}(t)$ is equal or not.
The case of (\ref{fise}) satisfies with $\Upsilon_{K_{p,q}}(t)=\Upsilon_{K}(s)+\Upsilon_{T_{p,q}}(t)$.
In order to explain the decomposition, we also use the figure in {\sc Figure}~\ref{38}.
The conditions (\ref{fise}) and (\ref{fise2}) correspond regions in {\sc Figure}~\ref{cc2} in this case of $(T_{3,7})_{3,35}$.
\begin{figure}[thbp]
\begin{overpic}
[width=.7\textwidth]
{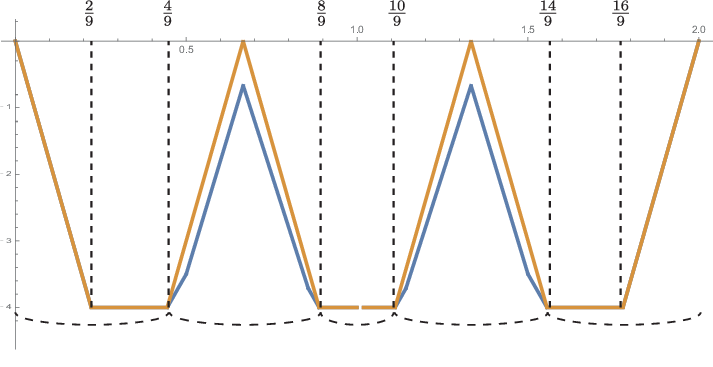}
\put(10,2){(\ref{fise})}
\put(31,2){(\ref{fise2})}
\put(47,2){(\ref{fise})}
\put(63.5,2){(\ref{fise2})}
\put(85,2){(\ref{fise})}

\put(14,35){$i=0$}
\put(46,35){$i=1$}
\put(78,35){$i=2$}
\end{overpic}
\caption{Regions for (\ref{fise}) and (\ref{fise2}) in the case of $(T_{3,7})_{3,35}$.}
\label{cc2}

\end{figure}

Theorem~\ref{cor} or \ref{s} correspond to the part (\ref{fise}) or (\ref{fise2})  respectively.
These will be proven in Section~\ref{13cond} and Section~\ref{13cond2} respectively.
\subsection{Modification of $\Phi(t,m)$}
In $z^i_{-1}<m\le z^{i}_{2g}$, the (local) behaviors of $\Phi(t,m)$ in the two cases of $2gp\le q$ and $(2g-1)p\le q<2gp$ are different.
The different points are 
$$\begin{cases}
z^1_0&i=0\\
z^i_0-\delta\text{ and }z^{i+1}_0&\text{ for }1\le i\le p-2\\
z^{p-1}-\delta&\text{ for }i=p-1.
\end{cases}$$
To compute the minimum of $\Phi(t,m)$ in the case of $(2g-1)p<q<2gp$,
we modify $\Phi(t,m)$ into $\Phi^i(t,m)$ as defined in the next paragraph.
For, $\Phi(t,m)$ does not satisfy this type of the cabling formula (\ref{cab}) any more because of the different behavior above.
On the other hand, the modified $\Phi^i(t,m)$ has the same local behavior as the one in the case of $2gp\le q$.
Hence, $\Phi^i(t,m)$ becomes applicable to the formula (\ref{cab}).

We put 
$$S_{K_{p,q}}^i=\begin{cases}
S_{K_{p,q}}\setminus\{z^1_0\}&i=0\\
\left(S_{K_{p,q}}\cup\{z^i_{0}-\delta\}\right)\setminus\{z^{i+1}_0\}&0<i<p-1\\
S_{K_{p,q}}\cup\{z^{p-1}_0-\delta\}&i=p-1
\end{cases}$$
and 
$$\Phi^i(t,m)=\#(S_{K_{p,q}}^i\cap [0,m))-tm/2+\begin{cases}
0&i=0\\
-1&0<i\le p-1.
\end{cases}$$
Two functions $\Phi^i(t,m)$ and $\Phi(t,m)$ coincide over $I_i^{\delta}$, while $\Phi^i(t,m)$ is the $-1$-shift of $\Phi(t,m)$ over the complement of $I_i^{\delta}$ when $1<i<p-1$.
{\sc Figure}~\ref{casecond} presents this.

\begin{figure}[htpb]
\begin{center}
{\unitlength 0.1in%
\begin{picture}(52.6500,27.5700)(5.3500,-30.2700)%
\put(6.0000,-4.0000){\makebox(0,0)[lb]{$\Phi(t,m)$}}%
%
\special{pn 8}%
\special{pa 800 600}%
\special{pa 800 1400}%
\special{dt 0.045}%
\special{pa 1400 600}%
\special{pa 1400 1400}%
\special{dt 0.045}%
\special{pa 1800 600}%
\special{pa 1800 1400}%
\special{dt 0.045}%
%
\special{pn 8}%
\special{pa 2800 600}%
\special{pa 2800 1400}%
\special{dt 0.045}%
\put(34.0000,-12.0000){\makebox(0,0)[lb]{$\cdots$}}%
\put(8.0000,-16.3500){\makebox(0,0){$z^i_{-1}$}}%
\put(12.7000,-17.0500){\makebox(0,0)[lb]{$\tiny{z^i_{0}-\delta}$}}%
\put(18.0000,-16.2500){\makebox(0,0){$z^i_0$}}%
\put(28.0000,-17.1500){\makebox(0,0)[lb]{$z^i_{1}$}}%
%
\special{pn 8}%
\special{pa 4400 600}%
\special{pa 4400 1400}%
\special{dt 0.045}%
%
\special{pn 8}%
\special{pa 5400 600}%
\special{pa 5400 1400}%
\special{dt 0.045}%
%
\special{pn 8}%
\special{pa 4800 600}%
\special{pa 4800 1400}%
\special{dt 0.045}%
\put(40.0000,-17.2000){\makebox(0,0)[lb]{$z^{i+1}_0-\delta$}}%
\put(48.0000,-16.6000){\makebox(0,0){$z^{i+1}_0$}}%
\put(54.0000,-16.4000){\makebox(0,0){$z^i_{2g}$}}%
%
\special{pn 8}%
\special{pa 800 2200}%
\special{pa 800 3000}%
\special{dt 0.045}%
\special{pa 1400 2200}%
\special{pa 1400 3000}%
\special{dt 0.045}%
\special{pa 1800 2200}%
\special{pa 1800 3000}%
\special{dt 0.045}%
%
\special{pn 8}%
\special{pa 2800 2200}%
\special{pa 2800 3000}%
\special{dt 0.045}%
\put(34.0000,-28.0000){\makebox(0,0)[lb]{$\cdots$}}%
%
\special{pn 8}%
\special{pa 800 530}%
\special{pa 1400 930}%
\special{fp}%
%
\special{pn 8}%
\special{pa 1400 930}%
\special{pa 1800 1130}%
\special{fp}%
%
\special{pn 8}%
\special{pa 4400 2200}%
\special{pa 4400 3000}%
\special{dt 0.045}%
%
\special{pn 8}%
\special{pa 5400 2200}%
\special{pa 5400 3000}%
\special{dt 0.045}%
%
\special{pn 8}%
\special{pa 4800 2200}%
\special{pa 4800 3000}%
\special{dt 0.045}%
%
\special{pn 8}%
\special{pa 4400 2560}%
\special{pa 4800 2760}%
\special{fp}%
%
\special{pn 8}%
\special{pa 4800 2760}%
\special{pa 5400 2960}%
\special{fp}%
\put(6.0000,-21.0000){\makebox(0,0)[lb]{$\Phi^i(t,m)$}}%
%
\special{pn 8}%
\special{pa 600 1400}%
\special{pa 5800 1400}%
\special{fp}%
\special{sh 1}%
\special{pa 5800 1400}%
\special{pa 5733 1380}%
\special{pa 5747 1400}%
\special{pa 5733 1420}%
\special{pa 5800 1400}%
\special{fp}%
%
\special{pn 8}%
\special{pa 600 3000}%
\special{pa 5800 3000}%
\special{fp}%
\special{sh 1}%
\special{pa 5800 3000}%
\special{pa 5733 2980}%
\special{pa 5747 3000}%
\special{pa 5733 3020}%
\special{pa 5800 3000}%
\special{fp}%
%
\special{pn 8}%
\special{pn 8}%
\special{pa 1400 2922}%
\special{pa 1394 2919}%
\special{fp}%
\special{pa 1381 2889}%
\special{pa 1379 2881}%
\special{fp}%
\special{pa 1372 2845}%
\special{pa 1371 2837}%
\special{fp}%
\special{pa 1368 2799}%
\special{pa 1367 2791}%
\special{fp}%
\special{pa 1365 2752}%
\special{pa 1365 2744}%
\special{fp}%
\special{pa 1365 2704}%
\special{pa 1365 2695}%
\special{fp}%
\special{pa 1367 2656}%
\special{pa 1368 2649}%
\special{fp}%
\special{pa 1372 2611}%
\special{pa 1373 2602}%
\special{fp}%
\special{pa 1379 2567}%
\special{pa 1381 2560}%
\special{fp}%
\special{pa 1394 2531}%
\special{pa 1400 2528}%
\special{fp}%
%
\special{pn 8}%
\special{pn 8}%
\special{pa 4800 1162}%
\special{pa 4794 1160}%
\special{fp}%
\special{pa 4782 1133}%
\special{pa 4780 1127}%
\special{fp}%
\special{pa 4774 1095}%
\special{pa 4773 1087}%
\special{fp}%
\special{pa 4769 1054}%
\special{pa 4768 1047}%
\special{fp}%
\special{pa 4766 1012}%
\special{pa 4766 1004}%
\special{fp}%
\special{pa 4765 969}%
\special{pa 4765 961}%
\special{fp}%
\special{pa 4766 926}%
\special{pa 4766 918}%
\special{fp}%
\special{pa 4768 883}%
\special{pa 4769 876}%
\special{fp}%
\special{pa 4773 843}%
\special{pa 4774 835}%
\special{fp}%
\special{pa 4780 803}%
\special{pa 4782 797}%
\special{fp}%
\special{pa 4794 770}%
\special{pa 4800 768}%
\special{fp}%
\put(46.5000,-9.8500){\makebox(0,0)[lb]{$1$}}%
\put(12.8000,-27.4000){\makebox(0,0)[lb]{$1$}}%
%
\special{pn 8}%
\special{pa 1400 1400}%
\special{pa 1400 1485}%
\special{fp}%
%
\special{pn 8}%
\special{pa 4400 1400}%
\special{pa 4400 1480}%
\special{fp}%
%
\special{pn 8}%
\special{pa 4800 1400}%
\special{pa 4800 1485}%
\special{fp}%
%
\special{pn 8}%
\special{pa 1800 1400}%
\special{pa 1800 1480}%
\special{fp}%
%
\special{pn 8}%
\special{pa 800 1400}%
\special{pa 800 1480}%
\special{fp}%
%
\special{pn 8}%
\special{pa 5400 1400}%
\special{pa 5400 1480}%
\special{fp}%
%
\special{pn 8}%
\special{pa 1800 2720}%
\special{pa 2800 2520}%
\special{fp}%
%
\special{pn 8}%
\special{pa 2800 2520}%
\special{pa 3110 2570}%
\special{fp}%
%
\special{pn 8}%
\special{pa 3090 2570}%
\special{pa 3400 2620}%
\special{dt 0.045}%
%
\special{pn 8}%
\special{pa 1800 1120}%
\special{pa 2800 920}%
\special{fp}%
%
\special{pn 8}%
\special{pa 2800 920}%
\special{pa 3110 970}%
\special{fp}%
%
\special{pn 8}%
\special{pa 3090 970}%
\special{pa 3400 1020}%
\special{dt 0.045}%
%
\special{pn 8}%
\special{pa 1400 2530}%
\special{pa 1800 2730}%
\special{fp}%
%
\special{pn 8}%
\special{pa 800 2530}%
\special{pa 1400 2930}%
\special{fp}%
%
\special{pn 8}%
\special{pa 4400 960}%
\special{pa 4800 1160}%
\special{fp}%
%
\special{pn 8}%
\special{pa 4800 760}%
\special{pa 5400 960}%
\special{fp}%
%
\special{pn 8}%
\special{pa 2800 1400}%
\special{pa 2800 1480}%
\special{fp}%
\end{picture}}%
\end{center}
\caption{The functions $\Phi(t,m)$ and $\Phi^i(t,m)$ (in case of $0<i<p-1$).}
\label{casecond}
\end{figure}

Here we consider the following conditions.

\begin{cond}
\label{co1}
$\underset{z^i_0-\delta< m\le z^{i+1}_0}{\min}\Phi(t,m)\le \underset{z^i_{-1}<m\le z^i_0-\delta}{\min}\Phi^i(t,m)$ holds.
\end{cond}
This condition is equivalent to 
$$\underset{z^i_{0}-\delta<m\le z^{i+1}_0}\min\Phi(t,m)=\underset{z^i_{-1}<m\le z^{i+1}_0}\min\Phi^i(t,m).$$
Actually, since $\Phi(t,m)=\Phi^i(t,m)$ holds over $m\in I^\delta_i$, we have
\begin{eqnarray*}
\underset{z^i_0-\delta<m\le z^{i+1}_0}\min\Phi(t,m)&=&\min\left\{\underset{z^i_{-1}<m\le z^i_0-\delta}\min\Phi^i(t,m),\underset{z^i_0-\delta<m\le z^{i+1}_0}\min\Phi(t,m)\right\}\\
&=&\underset{z^i_{-1}<m\le z_0^{i+1}}\min\Phi^i(t,m).
\end{eqnarray*}
\begin{cond}
\label{co2}
$\underset{z^i_0-\delta< m\le z^{i+1}_0}{\min}\Phi(t,m)\le \underset{z^{i+1}_0< m\le z^i_{2g}}{\min}\Phi^i(t,m)$ holds.
\end{cond}
In the same way as above this condition is equivalent to 
$$\underset{z^{i}_0-\delta<m\le z^{i+1}_0}\min\Phi(t,m)=\underset{z^{i}_0-\delta<m\le z^i_{2g}}\min\Phi^i(t,m).$$

Now, we prove the following Claim~\ref{ineqclaim}.
\begin{claim}
\label{ineqclaim}
Let $t,s$, and $i$ be parameters satisfying (\ref{setting}).

Let $m$ be an integer with $z^i_{-1}<m\le z^{i}_0$.
Then, for $0\le \nu\le 2g$ we have 
$$\Phi^i(t,m+\nu p)-\Phi^i(t,m)=\Phi_K(s,\nu).$$

Let $m$ be an integer with $z^i_{2g-1}=z^{i+1}_0-\delta<m\le z^i_{2g}$.
Then, for $0\le \nu\le 2g$ we have 
$$\Phi^i(t,m)-\Phi^i(t,m-\nu p)=-\Phi_K(2-s,\nu).$$
\end{claim}
\begin{proof}
Let $m$ be an integer with $z^i_{-1}<m\le z^{i}_0$.
Notice that (\ref{fundlem2}) is satisfied even if we use $\Phi^i(t,m)$ instead of $\Phi(t,m)$.
Then for any non-negative integer $\nu$ we have
\begin{eqnarray*}
\Phi^i(t,m+\nu p)-\Phi^i(t,m)&=&\#(S^i_{K_{p,q}}\cap [m,m+\nu p))-\nu pt/2\\
&=&i\nu+\#(S_K\cap [0,\nu))-\nu(i+\frac{s}{2})\\
&=&\#(S_K\cap [0,\nu))-\nu s/2=\Phi_K(s,\nu).
\end{eqnarray*}
We explain the second equality in Claim~\ref{ineqclaim}.
For any integer $j$ with $0\le j< \nu$, $\#(S^i_{K_{p,q}}\cap [m+jp,m+(j+1)p))$ is
$$\begin{cases}i+1\text{ points}&z^i_{j}\in S_{K_{p,q}}\\i\text{ points}&\text{otherwise.}\end{cases}$$
Using the following relationship $z_j^i\in S_{K_{p,q}}\Leftrightarrow j\in S_K$, we have $\#(S^i_{K_{p,q}}\cap [m,m+\nu p))=i\nu +\#(S_K\cap [0,\nu))$.

Let $m$ be an integer with $z^i_{2g-1}<m\le z^i_{2g}$.
Let $\bar{S}_K$ be the complement of $S_K$ in ${\mathbb Z}$.
In the same way as above, we have the following: 
\begin{eqnarray*}
\Phi^i(t,m)-\Phi^i(t,m-\nu p)&=&\#(S^i_{K_{p,q}}\cap [m-\nu p,m))-\nu pt/2\\
&=&i\nu+\#(S_K\cap [2g-\nu, 2g))-\nu(i+\frac{s}{2})\\
&=&\#(\bar{S}_K\cap [0,\nu))-\nu s/2\\
&=&\nu-\#(S_K\cap [0,\nu))-\nu s/2=-\Phi_K(2-s,\nu).
\end{eqnarray*}
\qed
\end{proof}

Next, we prove the following claim using Claim~\ref{ineqclaim}.
\begin{claim}
\label{claimimp}
The following holds:
\begin{itemize}
\item The inequality $s\ge \mu_K$ implies Condition~\ref{co1}.
\item The inequality $s\le 2-\mu_K$ implies Condition~\ref{co2}.
\end{itemize}
\end{claim}
\begin{proof}
If $s\ge \mu_K=\underset{1\le \nu\le 2g-1}{\min}\frac{2\varphi_K(\nu)}{\nu}$, then there exists an integer $\nu$ with $1\le \nu\le 2g-1$ satisfying $\Phi_K(s,\nu)\le 0$.
Then using the first formula in Claim~\ref{ineqclaim}, we have $\Phi^i(t,m+\nu p)-\Phi^i(t,m)=\Phi_K(s,\nu)\le 0$.
Thus applying this to $z^i_{-1}<m\le z^i_0-\delta$, we have
$$\underset{z^i_{-1}<m\le z^i_0-\delta}{\min}\Phi^i(t,m)\ge \underset{z^i_{-1}<m\le z^i_0-\delta}{\min}\Phi^i(t,m+\nu p)\ge \underset{z^i_0-\delta< m\le z^{i+1}_0}{\min}\Phi(t,m) $$
Therefore, Condition~\ref{co1} is satisfied.

In the same way, we prove the second one.
If $s\le 2-\mu_K$, then there exists an integer $\nu$ with $1\le \nu\le 2g-1$ satisfying $\Phi_K(2-s,\nu)\le 0$.
In the same way as above we have
$\Phi^i(t,m)-\Phi^i(t,m-\nu p)=-\Phi_K(2-s,\nu)\ge 0$.
$$\underset{z^{i+1}_0< m\le z^i_{2g}}{\min}\Phi^i(t,m)\ge \underset{z^{i+1}_0< m\le z^i_{2g}}{\min}\Phi^i(t,m-\nu p)\ge \underset{z^i_0-\delta< m\le z^{i+1}_0}{\min}\Phi(t,m).$$
Therefore, Condition~\ref{co2} is satisfied.
\qed
\end{proof}

\subsection{Proof of Theorem~\ref{cor}.}
\label{13cond}
In this section we prove Theorem~\ref{cor}.

\subsubsection{Suppose that $\mu_K\le s\le2-\mu_K$ and $1\le i\le p-2$ are satisfied.}
\label{case1}
Then, Claim~\ref{claimimp} means Condition \ref{co1} and \ref{co2} are satisfied.
Thus we have
$$\underset{m\in I_i^{\delta}}\min\Phi(t,m)=\underset{m\in I_i}\min\Phi^i(t,m).$$
By using the same argument as the case of $2gp\le q$,
$$\underset{m\in I_i}\min\Phi^i(t,m)=\underset{0\le m\le 2g}\min\Phi_K(s,m)+\mu_i.$$
Here $\mu_i$ is the same thing as (\ref{ladder}).
Thus we have
$$\Upsilon_{K_{p,q}}(t)=\Upsilon_{K}(s)+\Upsilon_{T_{p,q}}(t).$$

\subsubsection{Suppose that $s\le 2-\mu_K$ and $i=0$ are satisfied.}
\label{case2}
In this case, if $m\le q$, then $\Phi^0(t,m)=\Phi(t,m)$ is satisfied from the definition of $\Phi^0(t,m)$.
This implies for any $-p<m\le -\delta$, we have $\Phi(t,m)=-tm/2\ge 0=\Phi(t,0)$.
This means that Condition \ref{co1} is satisfied and 
$$\underset{m\in I_0^{\delta}}\min\Phi(t,m)=\underset{z^0_{-1}\le m\le z^1_{0}}\min\Phi(t,m).$$
Then $\underset{z^0_{-1}\le m\le z^1_0}\min\Phi(t,m)=\underset{-p<m\le 2gp}\min\Phi^0(t,m)=\underset{m\in I_0}\min\Phi^0(t,m)$ holds, hence using Equation~(\ref{phiformula}), and Claim~\ref{mucom}  for $\underset{m\in I_0}\min\Phi^0(t,m)$ we have
$$\Upsilon_{K_{p,q}}(t)=\Upsilon_K(s)+\Upsilon_{T_{p,q}}(t).$$

\subsubsection{Suppose that $\mu_K\ge s$ are satisfied and $i=p-1$.}
\label{case3}
From the symmetry of $\Upsilon_{K_{p,q}}(t)$ and exchanging $s$ and $2-s$, we have 
$$\Upsilon_{K_{p,q}}(t)=\Upsilon_K(s)+\Upsilon_{T_{p,q}}(t).$$
\qed

Considering Sections \ref{case1}, \ref{case2} and \ref{case3}, we can, therefore, finish the proof of Theorem~\ref{cor}.\qed
\subsection{Proof of Theorem~\ref{s}.}
\label{13cond2}
First, we prove symmetry of $\Upsilon_{K}^{tr}(s)$.
\begin{lem}
\label{23}
Let $K$ be an L-space knot. Then, we have $\Upsilon_{K}^{tr}(s)=\Upsilon_{K}^{tr}(2-s)$
\end{lem}
\begin{proof}
From the third of fundamental facts in Section~\ref{fs}, 
$$\#(\overline{S}_K\cap [0,\nu))=\#(S_K\cap [2g-\nu,2g)).$$
Using Lemma~\ref{fundlem1}, we have
\begin{eqnarray*}
\Upsilon_{K}^{tr}(s)&=&\max_{\nu\in \{1,2,\cdots, 2g-1\}}\tu_{K}(s,\nu)=\max_{\nu\in \{1,2,\cdots,2g-1\}}\tu_{K}(s,2g-\nu)\\
&=&-2\min_{\nu\in \{1,2,\cdots,2g-1\}}\left(g-\nu+\varphi_K(\nu)-s(2g-\nu)/2\right)-sg\\
&=&-2\min_{\nu\in \{1,2,\cdots,2g-1\}}\left(\varphi_K(\nu)-(2-s)\nu /2\right)-(2-s)g=\Upsilon_{K}^{tr}(2-s).
\end{eqnarray*}
\qed\end{proof}
From Section~\ref{theoremseven}, we give a proof of Theorem \ref{s}.
Let $t,s$, and $i$ be parameters satisfying (\ref{setting}).
\subsubsection{We suppose $0\le s<\mu_K$ and $1\le i\le p-1$.}
\label{theoremseven}
By using Equation~(\ref{ineq10}), we have $s<\mu_K\le 2-\mu_K$.
Then for any integer $\nu,m$ with $1\le \nu\le 2g-1$ and $z^i_{-1}< m\le z^{i}_0$, the following
$$\Phi^i(t,m+\nu p)-\Phi^i(t,m)=\Phi_K(s,\nu)>0$$
is satisfied.

Here, let $\mu_i^1(t)$ and $\mu_i^2(t)$ be the minimum values of $\Phi^i(t,m)$ over $z^i_{0}-\delta<m\le z^{i}_0$ and $z^i_{-1}<m\le z^i_0-\delta$ respectively.
Since $\Phi_K(s,2g)-\Phi_K(s,0)=\#(S_K\cap [0,2g))-gs=g(1-s)\ge0$, we obtain 
\begin{eqnarray*}
\underset{\nu\in \{0,1,\cdots,2g-1\}}{\min}\left(\underset{z^i_{\nu}-\delta<m\le z^i_\nu}{\min}\Phi^i(t,m)\right)&=&\underset{\nu\in \{0,1,\cdots,2g-1\}}{\min}\left(\Phi_K(s,\nu)+\mu_i^1(t)\right)\\
&=&\underset{\nu\in \{0,1,\cdots,2g\}}{\min}\Phi_K(s,\nu)+\mu_i^1(t),
\end{eqnarray*}
and
\begin{eqnarray*}
\underset{\nu\in \{1,2,\cdots,2g-1\}}{\min}\left( \underset{z^i_{\nu-1}<m\le z^i_{\nu}-\delta}{\min}\Phi^i(t,m)\right)&=&\underset{\nu\in \{1,\cdots,2g-1\}}\min(\Phi_K(s,\nu)+\mu_i^2(t))\\
&=&\underset{\nu\in \{1,2,\cdots, 2g-1\}}{\min}\Phi_K(s,\nu)+\mu_i^2(t).
\end{eqnarray*}

Since $s<2-\mu_K$ holds as mentioned as above, there exists an integer $\nu$ with $1\le\nu\le 2g-1$ such that $\frac{2\varphi_K(\nu)}{\nu}<2-s$ holds.
Hence, we have
$$\Phi^i(t,m)-\Phi^i(t,m-\nu p)=\#(\bar{S}_K\cap [0,\nu))-\nu s/2=-\Phi(2-s,\nu)>0.$$
Hence 
$$\underset{z^i_{0}<m\le z^{i+1}_{0}-\delta}\min\Phi^i(t,m)<\underset{z^{i+1}_0-\delta<m\le z^i_{2g}}\min\Phi^i(t,m).$$
Applying this, we obtain
\begin{eqnarray*}
&&\underset{z^i_0-\delta<m\le z^{i+1}_0}{\min}\Phi(t,m)=\underset{z^i_0-\delta<m\le z^{i+1}_0-\delta}{\min}\Phi^i(t,m)\\
&=&\min\left\{\underset{\nu\in \{0,1,\cdots,2g-1\}}{\min}\left(\underset{z^i_{\nu}-\delta<m\le z^i_\nu}{\min}\Phi^i(t,m)\right),\right.\\
&&\left.\underset{\nu\in \{1,2,\cdots,2g-1\}}{\min}\left( \underset{z^i_{\nu-1}<m\le z^i_{\nu}-\delta}{\min}\Phi^i(t,m)\right)\right\},
\end{eqnarray*}
Here Claim~\ref{mucom} can be applied in our case:
\begin{eqnarray*}
-2\mu_i^1(t)&=&-2\underset{z^i_0-\delta<m\le z^i_0}{\min}\Phi^i(t,m)=-2\min_{z^i_0-\delta<m\le z^i_0}\Phi_{p,q}(t,m)+2ig\\
&=&\Upsilon_{p,q}^{\delta,1}(t)+2ig+tg_{p,q}
\end{eqnarray*}
\begin{eqnarray*}
-2\mu_i^2(t)&=&-2\underset{z^i_{-1}<m\le z^i_0-\delta}{\min}\Phi^i(t,m)=-2\min_{z^i_{-1}<m\le z^i_0-\delta}\Phi_{p,q}(t,m)+2ig\\
&=&\Upsilon_{p,q}^{\delta,2}(t)+2ig+tg_{p,q}
\end{eqnarray*}
As a result, we can compute $\Upsilon$-invariants as follows:
\begin{eqnarray*}
&&\Upsilon_{K_{p,q}}(t)=-2\min_{z^i_0-\delta<m\le z^{i+1}_0}\Phi(t,m)-tg_{K_{p,q}}\\
&=&\max\left\{-2\min_{\nu=0,1\cdots,2g}\Phi_K(s,\nu)-tg_{K_{p,q}}-2\mu_i^1(t),\right.\\
&&\left.-2\min_{\nu\in \{1,2,\cdots,2g-1\}}\Phi_K(s,\nu)-tg_{K_{p,q}}-2\mu_i^2(t)\right\}\\
&=&\max\left\{\max_{\nu\in\{0,1\cdots,2g\}}\tu_K(s,\nu)+sg-t(pg+g_{p,q})+\Upsilon_{p,q}^{\delta,1}(t)+2ig+tg_{p,q},\right.\\
&&\left.\min_{\nu\in\{1,2,\cdots,2g-1\}}\tu_K(s,\nu)+sg-t(pg+g_{p,q})+\Upsilon_{p,q}^{\delta,2}(t)+2ig+tg_{p,q}\right\}\\
&=&\max\left\{\Upsilon_K(s)+\Upsilon_{p,q}^{\delta,1}(t),\Upsilon_{K}^{tr}(s)+\Upsilon_{p,q}^{\delta,2}(t)\right\}
\end{eqnarray*}

\subsubsection{We suppose $2-\mu_k<s\le 2$ and $0\le i\le p-2$.}
Reflecting $t$ as $t\mapsto 2-t$, other parameters are changed as $s\mapsto 2-s$ and $i\mapsto p-1-i$.
Then using the symmetry of $\Upsilon$-invariant and the right previous formula and Lemma~\ref{23}, we have
\begin{eqnarray*}
\Upsilon_{K_{p,q}}(t)&=&\Upsilon_{K_{p,q}}(t)\\
&=&\max\left\{\Upsilon_K(2-s)+\Upsilon_{p,q}^{\delta,1}(2-t),\Upsilon_{K}^{tr}(2-s)+\Upsilon_{p,q}^{\delta,2}(2-t)\right\}\\
&=&\max\left\{\Upsilon_K(s)+\Upsilon_{p,q}^{\delta,1}(2-t),\Upsilon_{K}^{tr}(s)+\Upsilon_{p,q}^{\delta,2}(2-t)\right\}.
\end{eqnarray*}
At this point, we completed the proof of Theorem~\ref{s}.
\qed

Here we prove a corollary stated in Section \ref{intro}.\\
{\bf Proof of Corollary \ref{pureinequality}.}
If $0<s<\mu_K$ or $2-\mu_K<s<2$ hold, then for all $\nu\in \{1,2,\cdots,2g-1\}$ $\Phi_K(s,\nu)>0=\Phi_K(s,0)=\Phi_K(s,2g)$ holds.
Thus $\Upsilon_K^{tr}(s)<\Upsilon_K(s)$ holds. 
Therefore, we have
\begin{eqnarray*}
\Upsilon_{K_{p,q}}(t)&\le &\max\left\{\Upsilon_K(s)+\Upsilon_{p,q}^{\delta,1}(t),\Upsilon_K(s)+\Upsilon_{p,q}^{\delta,2}(t)\right\}\\
&=&\Upsilon_K(s)+\max\left\{\Upsilon_{p,q}^{\delta,1}(t),\Upsilon_{p,q}^{\delta,2}(t)\right\}=\Upsilon_K(s)+\Upsilon_{T_{p,q}}(t)
\end{eqnarray*}
Since for any $1<\nu<2g-1$, the inequality $\Phi^i(t,m)<\Phi^i(t,m+\nu p)$ for $z^i_{-1}<m\le z^i_0-\delta$ holds.
This means that $\Upsilon_{K_{p,q}}(t)\le \Upsilon_K(s)+\Upsilon_{T_{p,q}}(t)$.
\hfill$\Box$\\

\section{The example $(T_{3,7})_{3,35}$.}
\label{examplecompute}
\subsection{Computation of $\Upsilon_{(T_{3,7})_{3,35}}$}
In Section~\ref{exin}, we checked the failure of $\Upsilon$-invariant formula of Theorem~\ref{cab} 
in the case of $(T_{3,7})_{3,35}$ by illustrating the difference of two graphs $\Upsilon_{T_{3,7}}(s)$ and $\Upsilon_{(T_{3,7})_{3,35}}(t)-\Upsilon_{T_{3,35}}(t)$
in {\sc Figure}~\ref{38}.
We, here, investigate in terms of Theorem~\ref{cor}.
Applying the cabling formula in Theorem~\ref{cor} and \ref{s} to this example,
we compute precise function of $\Upsilon_{(T_{3,7})_{3,35}}(t)$.
We put $K=T_{3,7}$.
The genera are computed as $g=g_{3,7}=6$, $g_{3,35}=34$, and $g(K_{3,35})=52$.
Then $S_K=\{0,3,6,7,9,10\}\cup{\mathbb Z}_{n\ge 12}$ holds, where  $K=T_{3,7}$.
When $\nu=0,1,2,\cdots,12$, the sequence $\varphi_K(\nu)$ is as follows:
$$\varphi_K(\nu):0,1,1,1,2,2,2,3,4,4,5,6,6.$$
Hence, we have $\mu_K=2/3$ and $\delta=2$.

First, we consider $2/3< t< 4/3$, namely this corresponds to the case of $i=1$ in Theorem~\ref{cor}.
Furthermore we assume $2+s=3t$ and $0<s<2/3$.
This means $2/3<t<8/9$.
Then we have
$$\tu_{T_{3,35}}(t,m)=-2\varphi_{T_{3,35}}(m)-(34-m)t.$$

Since $\varphi_{T_{3,35}}(34)=\varphi_{T_{3,35}}(35)=12$, we have
$$\Upsilon_{3,35}^{\delta,1}(t)=\underset{33<m\le 35}\max\tu_{T_{3,35}}(t,m)=\max\{-24,-24+t\}=-24+t.$$
Since $\varphi_{T_{3,35}}(33)=11$, we have
$$\Upsilon_{3,35}^{\delta,2}(t)=\underset{32<m\le 33}\max\tu_{T_{3,35}}(t,m)=-22-t.$$
If $0<s<2/3$, then we have
\begin{eqnarray}
\Upsilon_{K}^{tr}(s)&=&\max_{\nu\in\{1,\cdots,11\}}\left\{-2\varphi_K(\nu)-(6-\nu)s\right\}\nonumber\\
&=&\max_{\nu\in\{3,6,9\}}\left\{-2\varphi_K(\nu)-(6-\nu)s\right\}\label{secondeq}\\
&=&-2\varphi_K(3)-(6-3)s=-2-3s\nonumber
\end{eqnarray}
and while we have $\Upsilon_K(s)=-6s$.

Here we explain the second Equality (\ref{secondeq}).
We consider several candidates of functions which give the maximum in $\{-2\varphi_K(\nu)-(6-\nu)s|\nu=1,2,\cdots ,11\}$.
During the set of $N_i:=\{s\in \{0,1,\cdots, 11\}|\varphi_K(s)=i\}$ for $i\in {\mathbb N}$
the maximum function $-2\varphi_K(\nu)-(6-\nu)s$ is the one of the maximum $\nu$ in $N_i$.
This coincides with $S_K\cap[1,11]=\{3,6,7,9,10\}$.

Suppose that $\varphi_K(\nu-1)<\varphi_K(\nu)$.
Then since $-2\varphi_K(\nu-1)-(g-\nu+1)s>-2\varphi_K(\nu)-(g-\nu)s$ for any $0<s<2$.
The function for such $\nu\in S_K$ is not a candidate of the maximum function.
Thus we have only to consider $\{3,6,7,9,10\}-\{4,7,8,10,11\}=\{3,6,9\}$.

As a result, we have
$$\Upsilon^{\delta,1}_{3,35}(t)+\Upsilon_K(s)=-12-17t$$
and 
$$\Upsilon^{\delta,2}_{3,35}(t)+\Upsilon_{K}^{tr}(s)=-18-10t.$$
Hence, when $2/3<t<8/9$, the $\Upsilon_{K_{p,q}}(t)$ is the following:
\begin{eqnarray*}
\Upsilon_{K_{3,35}}(t)&=&\max\left\{-12-17t,-18-10t\right\}\\
&=&
\begin{cases}
-12-17t&2/3<t<6/7\\
-18-10t&6/7\le t<8/9.
\end{cases}
\end{eqnarray*}
Secondly, in $4/9<t<2/3$, applying (\ref{alter2}) in Theorem~\ref{cor}, we compute $\Upsilon_{K_{3,35}}(t)$ as follows:
\begin{eqnarray*}
\Upsilon_{K_{3,35}}(t)&=&\max\{-35t+(-12+18t),-34t+(-8+9t)\}\\
&=&\max\{-12-17t,-8-25t\}\\
&=&\begin{cases}-8-25t&4/9<t<1/2\\
-12-17t&1/2\le t<2/3.\end{cases}
\end{eqnarray*}
\section{Toward a further cabling formula}
Let $K$ be an L-space knot.
When positive relatively prime integers $p,q$ satisfy $q<(2g(K)-1)p$, the cable knot $K_{p,q}$ is not an L-space knot.
In this case, to compute the $\Upsilon$-invariant $\Upsilon_{K_{p,q}}$, we would require the different formula.
For example, consider the family $\Upsilon_{(T_{2,3})_{2,q}}$ for $q\in 2{\mathbb Z}+1$.
Then the paper can give the following equalities
$$\Upsilon_{(T_{2,3})_{2,2n+1}}(t)=\Upsilon_{T_{2,3}}(s)+\Upsilon_{T_{2,2n+1}}(t)\ \ \ (n>1),$$
where $2t\equiv s\bmod 2$ and $0\le s\le 2$.
Furthermore, since we have $\Delta_{(T_{2,3})_{2,3}}(t)=\Delta_{T_{3,4}}(t)$, we obtain
$$\Upsilon_{(T_{2,3})_{2,3}}(t)=\Upsilon_{T_{3,4}}(t).$$
Furthermore, we obtain $\Upsilon_{(T_{2,3})_{2,1}}(t)$ as the graph in {\sc Figure}~\ref{cable2321}.
This is due to Hedden's formula in \cite{Hed}.
This function coincides with  
$$\Upsilon_{(T_{2,3})_{2,1}}(t)=\Upsilon_{T_{3,4}}(t)-\Upsilon_{T_{2,3}}(t),$$
because $(T_{2,3})_{2,1}$ is $\nu^+$-equivalent to $T_{3,4}\#(-T_{2,3})$.
The $CFK^\infty((T_{2,3})_{2,1})$ is computed in \cite{Hed2}.
\begin{figure}[htpb]
\begin{center}
{\unitlength 0.1in%
\begin{picture}( 20.3100, 12.0600)( -0.2100,-17.7200)%
%
\special{pn 8}%
\special{pa 401 700}%
\special{pa 2010 700}%
\special{fp}%
\special{sh 1}%
\special{pa 2010 700}%
\special{pa 1943 680}%
\special{pa 1957 700}%
\special{pa 1943 720}%
\special{pa 2010 700}%
\special{fp}%
%
\special{pn 8}%
\special{pa 516 566}%
\special{pa 516 1772}%
\special{fp}%
\special{sh 1}%
\special{pa 516 1772}%
\special{pa 536 1705}%
\special{pa 516 1719}%
\special{pa 496 1705}%
\special{pa 516 1772}%
\special{fp}%
%
\special{pn 8}%
\special{pa 974 666}%
\special{pa 974 700}%
\special{fp}%
%
\special{pn 8}%
\special{pa 1890 700}%
\special{pa 1890 666}%
\special{fp}%
\put(18.9000,-6.3700){\makebox(0,0){$2$}}%
%
\special{pn 8}%
\special{pa 481 1616}%
\special{pa 516 1616}%
\special{fp}%
\put(3.4400,-16.1600){\makebox(0,0){$-\frac{4}{3}$}}%
%
\special{pn 8}%
\special{pa 516 700}%
\special{pa 974 1616}%
\special{fp}%
\special{pa 1203 1387}%
\special{pa 974 1616}%
\special{fp}%
\special{pa 1203 1387}%
\special{pa 1432 1616}%
\special{fp}%
\special{pa 1432 1616}%
\special{pa 1890 700}%
\special{fp}%
%
\special{pn 8}%
\special{pa 1432 666}%
\special{pa 1432 700}%
\special{fp}%
\end{picture}}%
\caption{$\Upsilon_{(T_{2,3})_{2,1}}$}
\label{cable2321}
\end{center}
\end{figure}
These equalities can be generalized in other cases of cable knots of torus knots.
For example, for $K=T_{2,5}$ and $g=2$
we have
$$\Upsilon_{K_{2,2n+5}}(t)=\Upsilon_{K}(s)+\Upsilon_{T_{2,2n+5}}(t)\ \ \ (n>1),$$
where $2t=s\bmod 2$ and $0\le s\le 2$.
However, $\Upsilon_{K_{2,7}}(t)$ does not equal to the $\Upsilon$-invariant of any L-space cable knot $(T_{a,b})_{p,q}$ of any torus knot $T_{a,b}$ with $q\ge 2g_{a,b}p$.
See Proposition~\ref{2527}.

Here we raise the following question.
\begin{que}
Let $K$ be an L-space knot.
Suppose that the integers $q,Q$ satisfy $q<(2g(K)-1)p<Q$.
Does there exist any method to compute the $\Upsilon_{K_{p,q}}(t)$ by using $\Upsilon_{K_{p,Q}}(t)$ and so on?
\end{que}
\section{Proofs of Proposition~\ref{torusknot} and Theorem~\ref{cabling}.}
In \cite{FK}, Feller and Krcatovich proved that the recurrence formula $\Upsilon_{T_{p,q}}(t)=\Upsilon_{T_{p,q-p}}(t)+\Upsilon_{T_{p,p+1}}(t)$.
By using this formula, they proved the following closed formula of $\Upsilon$-invariant of torus knots.
\begin{prop}[Proposition 2.2 in \cite{FK}]
\label{torusknotformula}
Let $a_i$ be the same coefficient defined  in (\ref{continued}) and $p_i$ the denominator of $[a_i,a_{i+1},\cdots,a_n]$.
Then we have
\begin{equation}\label{tkf}\Upsilon_{T_{p,q}}(t)=\sum_{i=1}^na_i\Upsilon_{T_{p_i,p_i+1}}(t).\end{equation}
\end{prop}
Note that the formula depends on the way of taking the continued fraction in general, but it does not depend on the way to take the non-negative integral continued fraction expansions $q/p=[a_i,\cdots,a_n]$, i.e., $a_i\ge 0$ for any $i$.
Here we prove Proposition~\ref{torusknot} by using the formula (\ref{tkf}).

\begin{proof}
From the torus knot formula, we immediately have
$$I(T_{p,q})=\sum_{i=1}^na_iI(T_{p_i,p_i+1}).$$
Comparing the first derivative of (\ref{tkf}) at $t=0$, we have
\begin{equation}
\label{theeq}
(p-1)(q-1)=\sum_{i=1}^na_ip_i(p_i-1).
\end{equation}
The direct computation for $T_{p,p+1}$ implies the following:
$$I(T_{p,p+1})=-\frac{p^2-1}{3}.$$
Thus, we have
$$I(T_{p,q})=-\frac{1}{3}\sum_{i=1}^na_i(p_i^2-1)=-\frac{1}{3}\sum_{i=1}^n(a_ip_i(p_i-1)-a_i+a_ip_i).$$
Since $p_{i-1}=a_ip_i+p_{i+1}$, 
$$\sum_{i=1}^na_ip_i=\sum_{i=1}^n(p_{i-1}-p_{i+1})=q+p_1-p_n=q+p-1.$$
Thus  using (\ref{theeq}) we get the following:
$$I(T_{p,q})=-\frac{1}{3}\left((p-1)(q-1)-\sum_{i=1}^na_i+q+p-1\right)=-\frac{1}{3}\left(pq-\sum_{i=1}^na_i\right).$$
\hfill$\Box$
\end{proof}

Next, we prove Theorem~\ref{cabling} using Theorem \ref{main}.\\
\begin{proof}
We put $L'=L_{n-1}$.
First we obtain the equality:
$$\int_0^2\Upsilon_{L'}(pt)dt=\int_0^{2p}\Upsilon_{L'}(s)\frac{1}{p}ds=p\int_0^2\Upsilon_{L'}(s)\frac{1}{p}ds=I(L').$$
This equality can be justified by regarding $\Upsilon_K(pt)$ as an extended function in such a way that
we extend $\Upsilon_K(t)$ naturally to the periodic function over ${\mathbb R}$ with the period $2$.
Using Theorem~\ref{main} and this computation we have
\begin{eqnarray*}
I(L)&=&\int_0^2(\Upsilon_{L'}(pt)+\Upsilon_{T_{p_n,q_n}}(t))dt=I(L')+I(T_{p_n,q_n}).
\end{eqnarray*}
By iterating this relationship we have
$$I(L)=I(K)+\sum_{i=1}^nI(T_{p_i,q_i}).$$
\hfill$\Box$
\end{proof}
Here we give an application.
\begin{prop}
\label{2527}
Let $L$ be any iterated torus knot 
$$(\cdots(T_{p_1,q_1})_{p_2,q_2})\cdots)_{p_n,q_n}$$
such that $(p_i,q_i)$ is any relatively prime integers satisfying $0<2g_{i}p_{i+1}\le q_{i+1}$ for $1\le i\le n-1$, where 
$g_i$ the Seifert genus of the iterated torus knot
$$(\cdots(T_{p_1,q_1})_{p_2,q_2})\cdots)_{p_i,q_i}.$$
Then $(T_{2,5})_{2,7}$ is not knot concordant to $L$. 
\end{prop}
\begin{proof}
Let $K$ be $(T_{2,5})_{2,7}$.
We give computation $I(K)=-17/3$ by using the formula in Theorem~\ref{s}.
If the assertion is not satisfied, then $I(K)$ is
a linear combination of several $I(T_{p,p+1})$'s with positive integer coefficients.
Since $I(T_{2,3})=-1$, $I(T_{3,4})=-8/3$, $I(T_{4,5})=-5$ and $I(T_{p,p+1})\le -8$ for $p\ge 5$,
the following is unique candidate:
$$I(K)=3I(T_{2,3})+I(T_{3,4}).$$
The denominator $p_1$ of $q_1/p_1$ is $2$ or $3$.
Since the continued fractions of $q_1/p_1$ are $1+1/2$, $2+1/2$, $3+1/2$, $1+1/3$ or  $1+1/(1+1/2)$.
These are 
$$3/2, 5/2, 7/2, 4/3\text{ and }5/3$$
respectively.
If the next rational $q_2/p_2$ satisfies $q_2/p_2\ge 2g_2=2$, $4$, $6$, $6$ and $8$ respectively.
Since we cannot make either of the above five fractions with $q_2/p_2\ge 4$, we have $q_1/p_1=3/2$.
The remaining value is 
$$I(K)-I(T_{2,3})=2(T_{2,3})+I(T_{3,4}).$$
The candidate of $q_2/p_2$ is $5/2$ only.
The next condition is $q_3/p_3\ge 4$.
We cannot make either of any five fractions with $q_3/p_3\ge 4$.
Therefore $3I(T_{2,3})+I(T_{3,4})$ is not any integral values of such iterated torus knots.
Since $I$ is a knot concordance invariant, $K$ is not knot concordant to any such iterated torus knot.
\hfill$\Box$
\end{proof}

	\end{document}